\documentclass{amsart}

\usepackage{latexsym,amssymb,amsmath, amsfonts}
\usepackage{amsthm}
\usepackage{amscd}

\numberwithin{equation}{section}
\usepackage{hyperref}
\hypersetup{    colorlinks = true,}
\newcommand{\R}{\mathbb{R}}

\renewcommand{\le}{\leqslant}
\renewcommand{\ge}{\geqslant}
\renewcommand{\leq}{\leqslant}
\renewcommand{\geq}{\geqslant}

\newcommand{\be}{\begin{equation}}
\newcommand{\en}{\end{equation}}
\newcommand{\ee}{\end{equation}}

\DeclareMathOperator{\supp}{supp}

\newcommand{\bt}{\begin{theorem}}
\newcommand{\et}{\end{theorem}}
\newcommand{\bp}{\begin{proof}}
\newcommand{\ep}{\end{proof}}
\newcommand{\bc}{\begin{cor}}
\newcommand{\ec}{\end{cor}}
\newcommand{\bl}{\begin{lemma}}
\newcommand{\el}{\end{lemma}}
\newcommand{\bprop}{\begin{prop}}
\newcommand{\eprop}{\end{prop}}

\newtheorem{theorem}{Theorem}[section]

\newtheorem{lemma}[theorem]{Lemma}

\newtheorem{prop}[theorem]{Proposition}

\newtheorem{cor}[theorem]{Corollary}

\numberwithin{theorem}{section} \numberwithin{definition}{section}

\theoremstyle{definition}

\author[J. Hickman]{Jonathan Hickman}
\address{School of Mathematics, The University of Edinburgh, Edinburgh EH9 3JZ, UK}
\email{jonathan.hickman@ed.ac.uk}
\author[F. Linares]{Felipe Linares}
\address{IMPA, Estrada Dona Castorina 110, Rio de Janeiro 22460-320, RJ Brasil}\email{linares@impa.br}
\author[O. G. Ria\~no]{Oscar G. Ria\~no}
\address{IMPA, Estrada Dona Castorina 110, Rio de Janeiro 22460-320, RJ Brasil}\email{ogrianoc@impa.br}
\author[K. M. Rogers]{Keith M. Rogers}
\address{Instituto de Ciencias Matem\'aticas CSIC-UAM-UC3M-UCM, 28049 Madrid, Spain}\email{keith.rogers@icmat.es}
\author[J. Wright]{James Wright}
\address{School of Mathematics and Maxwell Institute for Mathematical Sciences, The University of Edinburgh, Edinburgh EH9 3JZ, UK}\email{J.R.Wright@ed.ac.uk}
\thanks{Partially supported by CNPq and FAPERJ/Brazil, the MINECO grants SEV-2015-0554 and MTM2017-85934-C3-1-P, and the ERC grant 834728.}

\date{}
\title{On a higher dimensional version of the Benjamin--Ono equation}
\keywords{Benjamin--Ono equation, Strichartz estimates, local well-posedness}

\begin{document}

\begin{abstract} 
We consider a higher dimensional version of the Benjamin--Ono equation,
$\partial_t u -\mathcal{R}_1\Delta u+u\partial_{x_1} u=0$, where  $\mathcal{R}_1$ denotes the Riesz transform with respect to the first coordinate.
We first establish sharp space--time estimates for the associated linear equation. These estimates enable us to show that the initial value problem for the nonlinear equation is locally well-posed in $L^2$-Sobolev spaces $H^{s}(\mathbb{R}^d)$, with $s>5/3$ if $d=2$ and $s>d/2+1/2$ if $d\ge 3$. We also provide ill-posedness results.
\end{abstract}

\maketitle

\section{Introduction}\label{intro}

With $d\ge 2$, we consider the initial value problem for a higher dimensional version of the Benjamin--Ono equation;
\begin{equation}\label{HBO-IVP}\tag{HBO}
\begin{cases}
  \partial_t u - \mathcal{R}_1\Delta u+u\partial_{x_1} u=0,\hskip 15pt x\in \R^d,\,  t\in \R, \\
  u(x,0)= u_0.
  \end{cases}
\end{equation}
Here $\mathcal{R}_1$ denotes the Riesz transform with respect to the first coordinate~$x_1$ and~$\Delta$ denotes the Laplacian with respect to the full spatial variable $x\in\R^d$.  If $u(\cdot,t):\mathbb{R}^d\to \mathbb{R}$ solves \eqref{HBO-IVP} at a certain time $t$, then $-u(\cdot,-t)$ solves the analogous problem with~$-\mathcal{R}_1$ replaced by $\mathcal{R}_1$, and so the sign in the equation is not important.

Taking $d=1$, the Riesz transform coincides with the Hilbert transform, and so we recover the extensively studied Benjamin--Ono equation; see~\cite{BP, CW, Ponce1991} and the references therein. The equation maintains its physical interest with $d=2$; see for example~\cite{A,PS,VS} and the references therein. Both the cases $d=1$ and~$2$ have been used to model one dimensional internal waves in stratified fluids in $\mathbb{R}^3$ (with a vertical discontinuity in the density of the fluid). Indeed, Mari\c s~\cite{M} found that solitary wave solutions can still propagate when $d=2$, of the form $u(x_1,x_2,t)=\varphi(x_1- ct,x_2)$ (see also \cite{EP}).

With $d=1$, the available local well-posedness theory has been based on compactness methods. Indeed, Molinet, Saut and Tzvetkov~\cite{molin} proved that the problem cannot be solved in $L^2$-Sobolev spaces $H^{s}$ by Picard iteration. We will show that this remains true with $d\ge 2$, and so compactness methods will also play a role here. In higher dimensions, the $d=2$ case presents the most mathematical difficulties, at least with the techniques that we will employ.

Combining the Kato--Ponce commutator estimate~\cite{KP} with Gronwall's inequality, one can show that smooth solutions of ~\eqref{HBO-IVP} satisfy 
\begin{equation}\label{EE-HBO}
\sup_{[0,T]}\|u(t)\|_{H^s} \le \|u(0)\|_{H^s} \,\exp \Big(c\int_0^T \|\nabla u(t)\|_{L^{\infty}}\,dt\Big)
\end{equation}
for all $T>0$. Thus, if we could control the argument of the exponential function
by the $H^s$-norm, we could argue by compactness in order to establish the existence of solutions with less regularity. If this were to be done using Sobolev embedding, the required order of regularity would be $s>d/2+1$. However this would fail to take advantage of the additional {\it dispersion} of the higher dimensional equation. 

Instead we follow the idea introduced by Koch and Tzvetkov~\cite{KochT} to study the local well-posedness of the one dimensional Benjamin--Ono equation.   Roughly this consists of using Strichartz estimates rather than Sobolev embeddings. Extensions of this method were given by Kenig and K\"onig~\cite{KenigKo}. This will be the starting point in our analysis and so we first establish {\it Strichartz estimates} for the linear problem.

\subsection{Main Results} Our first result is a sharp  Strichartz estimate with $d\ge 2$. More precisely, we  consider
the linear equation
\begin{equation}\label{LHBO}\tag{LHBO}
\begin{cases}
 \partial_t v =\mathcal{R}_1\Delta v,\hskip 15pt x\in \R^d,\,  t\in \R,\\
 v(x,0)=f(x),
 \end{cases}
\end{equation}
where we can write $\mathcal{R}_1=-(-\Delta)^{-1/2}\partial_{x_1}$ with the fractional Laplacian defined as usual in terms of the Fourier transform 
$(-\Delta)^{s}f:=(|\cdot|^{2s}\widehat{f}\,)^{\vee}$. Smooth solutions to~\eqref{LHBO} can be similarly written as $$v(x,t)=e^{t\mathcal{R}_1\Delta} f(x):= \frac{1}{(2\pi)^d}\int_{\R^d} e^{i\xi\cdot x}e^{ i\xi_1|\xi|t}\widehat{f}(\xi)\,d\xi.$$  
We will prove the following theorem for data in homogeneous Sobolev spaces, with norm given by $\|f\|_{\dot{H}^s}:=\|(-\Delta)^{s/2}f\|_{L^2}$. The estimate is sharp with respect to the regularity and the Lebesgue exponents. Kenig, Ponce and Vega \cite[Theorem 2.4]{KPV} proved the analogous estimate for $d=1$ with data in $L^2(\R)$.

\begin{theorem}\label{main} Let $d\ge 3$ and $q<\infty$. Then there is a constant $C\equiv C(d,q,r)$ such that
$$
\|e^{t\mathcal{R}_1\Delta} f\|_{L_t^{r}(\R,L_x^q(\R^{d}))}\le C\,\| f\|_{\dot{H}^{s}(\R^d)}
$$
holds for all $f\in \dot{H}^{s}(\R^d)$  if and only if $\frac{2}{q}+\frac{2}{r}\le 1$ and $s=d(\frac{1}{2}-\frac{1}{q})-\frac{2}{r}$.\\
Moreover, with $d=2$, $$
\|e^{t\mathcal{R}_1\Delta} f\|_{L_t^{r}(\R,L_x^q(\R^{2}))}\le C\,\| f\|_{\dot{H}^{s}(\R^2)}
$$
holds for all $f\in \dot{H}^{s}(\R^2)$  if and only if $\frac{10}{q}+\frac{12}{r}\le 5$ and $s=1-\frac{2}{q}-\frac{2}{r}$.
\end{theorem}

We will also prove a {\it local smoothing} estimate. This kind of estimate was first proved by Kato~\cite{K} for the Korteweg--de Vries equation, and by Vega~\cite{V} and Constantin--Saut~\cite{CS} for the Schr\"odinger equation. With $d=1$, the following estimate follows from an identity first observed  by  Kenig, Ponce and Vega~\cite[Lemma 2.1]{KPV}.

\begin{theorem}\label{smoothing} Let $d\ge 2$ and $\alpha >1/2$. Then there is a constant $C\equiv C(d,\alpha)$ such that
$$
\int_{\R^{d+1}}|e^{t\mathcal{R}_1\Delta} f(x)|^2 \frac{dxdt}{(1+|x|^2)^\alpha}\le C\,\|f\|_{\dot{H}^{-1/2}(\R^d)}.
$$
\end{theorem}

 As remarked in the introduction,  one can prove that~\eqref{HBO-IVP} is locally well-posed in inhomogeneous Sobolev spaces 
 $H^s(\mathbb{R}^d)$ for $s>d/2+1$ using standard compactness methods. These spaces are defined in the same way as the homogeneous spaces but with  $(-\Delta)^{s/2}$ replaced by $J^s:=(\mathrm{I}-\Delta)^{s/2}$. Using the Strichartz estimates  we make the following improvement of the standard result. The Sobolev space $W^{1,\infty}$ is defined as usual with norm $\|f\|_{W^{1,\infty}}:=\|f\|_{L^{\infty}}+\|\nabla f\|_{L^{\infty}}$.  
  
\begin{theorem}\label{imprwellpos}
Let $s>s_d$ where $s_d := d/2+1/2$ for $d\ge 3$ and $s_2 := 5/3$. Then, for any $u_0\in H^s(\mathbb{R}^d)$, there exist a time $T=T(\left\|u_0\right\|_{H^s})$ and a unique solution~$u$ to~\eqref{HBO-IVP} that belongs to
\begin{equation*}
    C\big([0,T);H^s(\mathbb{R}^d)\big)\cap L^1\big([0,T);W^{1,\infty}(\mathbb{R}^d)\big).
\end{equation*}
Moreover, the flow map $u_0 \mapsto u(t)$ is continuous from $H^s(\mathbb{R}^d)$ to $H^s(\mathbb{R}^d)$.
\end{theorem}

As mentioned above, we will show that the flow map $u_0 \mapsto u(t)$ is not of class $C^2$ for any $s\in \mathbb{R}$.  In particular, this implies
that~\eqref{HBO-IVP} cannot be solved using
the Duhamel formulation combined with the contraction mapping principle.

\begin{theorem}\label{illpossed}

Let $s\in \mathbb{R}$. Then~\eqref{HBO-IVP} does not admit a  solution $u$  such that the flow map
$u_0 \mapsto u(t)$
is $C^2$-differentiable from $H^s(\mathbb{R}^d)$ to $H^s(\mathbb{R}^d)$. 
\end{theorem}

With $d=2$, we use the existence of solitary wave solutions~\cite{M} to show that the flow map cannot be uniformly continuous in $L^2(\R^2)$.

\begin{prop}\label{illpossedl2}
Let $d=2$. Then~\eqref{HBO-IVP} does not admit a solution $u$ such that the flow map $u_0\mapsto u(t)$ is uniformly continuous from $L^2(\mathbb{R}^2)$ to $L^2(\mathbb{R}^2)$.
\end{prop}

Some remarks are in order:
\begin{enumerate}
\item If $u$ solves~\eqref{HBO-IVP}, then so does the scaled version $u_{\lambda}$ defined by
$$
u_{\lambda}(x,t):=\lambda u(\lambda x,\lambda^2 t), 
$$
for any positive $\lambda$. On the other hand, one can calculate that 
\begin{equation*}
    \left\|u_{\lambda}(\cdot,t)\right\|_{\dot{H}^s}=\lambda^{1-d/2+s}\left\|u(\cdot,\lambda^2 t)\right\|_{\dot{H}^s}.
\end{equation*}
 As a consequence, the scale-invariant regularity for~\eqref{HBO-IVP}  is $s=d/2-1$. 
 In particular, the $d=2$ problem is $L^2$-critical.  Thus our results are far from reaching the regularity  suggested by scaling.

\item The higher dimensional Benjamin--Ono equation has a Hamiltonian structure. Formally, there are at least three quantities conserved by the flow:
    \begin{align*}
    I(u)&:=\int u(x,t) \, dx, \\
    M(u)&:=\int u^2(x,t) \, dx, \\
    H(u)&:=\int \left|(-\Delta)^{1/4} u(x,t)\right|^2- \frac{1}{3}u^3(x,t) \, dx. 
    \end{align*}
Note that  $H^{1/2}(\mathbb{R}^d)\hookrightarrow L^3(\mathbb{R}^d)$ by Sobolev embedding when $d\le 3$, so that $H(u)$ is well-defined in those cases. Unfortunately our local well-posedness results require too much regularity to be able to take advantage of this.
\item  For the one dimensional Benjamin--Ono equation, Tao~\cite{TaoBO} introduced a gauge transformation which allowed him to establish local and global results in $H^1(\R)$. In the end, it was possible to go all the way to $L^2(\R)$ using this gauge transformation; see~\cite{IoneKenig} and~\cite{molinetPilod}. We do not know if there is
such a gauge transformation for the higher dimensional Benjamin--Ono equation. 
\end{enumerate}

We will begin with the linear equation~\eqref{LHBO}. In the following section, we prove the local smoothing estimate of Theorem~\ref{smoothing}. In the third section, we prove sharp decay rates for certain relevant oscillatory integrals, and in the fourth section we use a well-known variant of an argument due to Tomas~\cite{T} to establish the Strichartz estimates of Theorems~\ref{main}. We then proceed to  consider the initial value problem for the nonlinear equation~\eqref{HBO-IVP}. The fifth section is devoted to proving Theorem~\ref{imprwellpos}.  We conclude the paper with an appendix where we show the ill-posedness results stated in Theorem~\ref{illpossed} and Proposition~\ref{illpossedl2}.


\section{Proof of Theorem~\ref{smoothing}}

We will require the following trace estimate for domains with boundaries that can be written as graphs of measurable functions.

\begin{theorem}\label{trace theorem} Let $d\ge 2$, $\alpha>1/2$ and $\gamma:\R^{d-1}\to \R$ be any measurable function. Then there is a constant $C\equiv C(d,\alpha)$, independent of $\gamma$, such that
$$
\|Tf\|_{L^2(\R^{d-1})}\le C\,\|f\|_{H^\alpha(\R^d)},
$$
where  $T: f\mapsto f\big(\gamma(\cdot),\cdot\big)$ is the trace operator.
\end{theorem}

\begin{proof}
Writing $x=(x_1,x')$ and $d\sigma(x_1,x') =\delta\big(x_1-\gamma(x')\big)$, by duality, it will suffice to prove
\begin{equation*}
\int_{\R^{d}} |\widehat{h d\sigma}(\xi)|^2 \frac{d\xi}{(1+|\xi|^2)^\alpha} \le C\,\|h\|^2_{L^2(d\sigma)}.
\end{equation*}
Writing $g(x') := h(\gamma(x'),x')$, this can be rewritten as
\begin{equation}\label{dual}
\int_{\R^{d}} \Big|\int_{\R^{d-1}}g(x')e^{-i[\xi_1\gamma(x')+\xi'\cdot x']}dx' \Big|^2 \frac{d\xi}{(1+|\xi|^2)^\alpha} \le C\,\|g\|_{L^2(\R^{d-1}}).
\end{equation}

Thus, by squaring out and Fubini's theorem, it would suffice to prove
$$
\int_{\R^{d-1}}\int_{\R^{d-1}}g(x') \overline{g(y')}\int_{\R^{d}}  e^{-i[\xi_1(\gamma(x')-\gamma(y'))+\xi'\cdot(x'-y')]} \frac{d\xi}{(1+|\xi|^2)^\alpha} dx'dy' \le C\,\|g\|^2_{L^2(\R^{d-1})}.
$$
The inner integral is a Fourier transform, and so the left-hand side of this can be written as
$$
(2\pi)^d\int_{\R^{d-1}}\int_{\R^{d-1}}g(x') \overline{g(y')} J^{-2\alpha}\big(\gamma(y')-\gamma(x'), y'-x'\big) dx'dy',
$$
where the Bessel potential $J^{-2\alpha}$ is well-known to  satisfy
$$
|J^{-2\alpha}(x)|\le  C_{d,\alpha}\left\{  
\begin{array}{lcl}
|x|^{-(d-2\alpha)},  & \mbox{when} & 
|x|\le 1 \\
|x|^{-d},  & \mbox{when} &  
|x|>1; \\
\end{array} \right.
$$
see, for example, \cite[Section 6.1.2]{Graf}.
In particular, by simply ignoring the $x_1$-variable, we have that 
$$
|J^{-2\alpha}\big(\gamma(y')-\gamma(x'),y'-x'\big)|\le \phi_{\alpha}(x'-y')
$$
where $\phi_{\alpha}:\R^{d-1}\to \R$ is an integrable function as long as $\alpha>1/2$. Thus it remains to prove
$$
\int_{\R^{d-1}}\int_{\R^{d-1}}|g(x') g(y')| \phi_\alpha (x'-y')\,dx'dy' \le C\,\|g\|^2_{L^2({\color{red}\R^{d-1}})},
$$
which follows by the Cauchy--Schwarz inequality and then Young's inequality.
\end{proof}

\begin{proof}[Proof of Theorem~\ref{smoothing}] By Plancherel's identity, 
$$
(2\pi)^{2d+1}\int|e^{t\mathcal{R}_1\Delta} f(x)|^2 dt=\int \Big| \int_{\R^d} \widehat{f}(\xi)\delta(\xi_1|\xi|-\tau)e^{ix\cdot\xi}\,d\xi\Big|^2d\tau,
$$
so by writing $\xi=(\xi_1,\xi')$ and changing variables $u=\xi_1|\xi|$, we see that
\begin{align*}
(2\pi)^{2d+1}\int|e^{t\mathcal{R}_1\Delta} f(x)|^2 dt&=\int \Big| \int_{\R^d} \widehat{f}(\xi)\delta(u-\tau)e^{ix\cdot\xi}\,\frac{|\xi|dud\xi'}{2\xi_1^2+|\xi'|^2}\Big|^2d\tau\\
&\!\!\!\!\!\!\!\!\!\!\!\!\!\!\!\!=\int \left|\int_{\R^{d-1}} \widehat{f}(\gamma_\tau(\xi'),\xi'\,)e^{i[x_1\gamma_\tau(\xi')+x'\cdot \xi']}\,\frac{\sqrt{\gamma^2_\tau(\xi')+|\xi'|^2}}{2\gamma_\tau^2(\xi')+|\xi'|^2}d\xi'\right|^2d\tau,
\end{align*}
where $\gamma^2_\tau(\xi') := \sqrt{\tau^2+(\frac{1}{2}|\xi'|^2)^2}-\frac{1}{2}|\xi'|^2$. Now by integrating both sides with respect to $(1+|x|^2)^{-\alpha}dx$, applying Fubini's theorem and the trace theorem in dual form~\eqref{dual}, we obtain 
$$
\int_{\R^{d+1}}|e^{t\mathcal{R}_1\Delta} f(x)|^2 \frac{dxdt}{(1+|x|^2)^\alpha}
\le C\!\int \int_{\R^{d-1}} \left| \widehat{f}(\gamma_\tau(\xi'),\xi')\frac{\sqrt{\gamma^2_\tau(\xi')+|\xi'|^2}}{2\gamma_\tau^2(\xi')+|\xi'|^2}\right|^2d\xi'd\tau.
$$
Finally, we reverse the change of variables so that
$$
\int_{\R^{d+1}}|e^{t\mathcal{R}_1\Delta} f(x)|^2 \frac{dxdt}{(1+|x|^2)^\alpha}
\le C\!\int_{\R^d} |\widehat{f}(\xi)|^2\frac{|\xi|d\xi}{2\xi_1^2+|\xi'|^2},
$$
which is slightly better than the desired bound.
\end{proof}

\section{The oscillatory integral}

An analogous estimate to the following with $d=1$ was proven by Ponce and Vega~\cite[Corollary 2.3]{PV} with decay of order $|t|^{-1/2}$.  

\begin{prop}\label{thelemma} Let $d\ge 3$ and $\psi:\mathbb{R}\to \mathbb{R}$ be a  Schwartz function supported in~$[1/2,2]$. Then there is a constant $C\equiv C(d,\psi)$ such that
$$
\Big|\int_{\R^d} \psi(|\xi|)\, e^{i[t\xi_1|\xi| +x\cdot\xi]}d\xi\Big|\le C|t|^{-1},\quad x\in \R^d.
$$
Moreover, with $d=2$, 
$$
\Big|\int_{\R^2} \psi(|\xi|)\, e^{i[t\xi_1|\xi| +x\cdot\xi]}d\xi\Big|\le C|t|^{-5/6},\quad x\in \R^2.
$$
Both rates of decay are optimal.
\end{prop}

\begin{proof} We first remark that the decay rates cannot be improved as otherwise the estimates of Theorem~\ref{main} would hold in larger ranges, which is not possible.

To prove the estimates, we write the integral as a two-fold iterated integral in radial and spherical
coordinates;
\begin{equation}\label{osci1}
    \begin{aligned}
    I(x,t):=\int_{\R^d} \psi(|\xi|)\, e^{i[t\xi_1|\xi| +x\cdot\xi]}d\xi &\ = \ \int^\infty_0 \rho(r) \Bigl[ \int_{{\mathbb{S}}^{d-1}} 
e^{i[t r^2 \omega_1  + x\cdot r \omega ]}   d\sigma(\omega) \Bigr] \, dr \\
&\ = \ \int_0^\infty \rho(r) {\widehat{\sigma}}(y(r)) dr
    \end{aligned}
\end{equation}
where $\rho(r) := \psi(r) r^{d-1}$ and $y(r) :=: y_{x,t}(r) := (t r^2 + x_1 r, x_2 r, \ldots, x_d r)$. We consider first the easier higher dimensional cases, when $d\ge 4$, then $d=3$, finally treating the harder $d=2$ case.


{\bf Dimensions $d\ge 4$}. Consider the second iterated integral in~\eqref{osci1} where the integral over the
sphere ${\mathbb{S}}^{d-1}$ is performed first and computed as
a Fourier transform ${\widehat{\sigma}}(y(r))$. We use the
stationary phase estimate
$$|{\widehat{\sigma}}(y(r))| \le C_d \min(1, |y(r)|^{-(d-1)/2}) \le
C \min(1, |t r^2 + x_1 r|^{-(d-1)/2})$$ (see, for example, \cite[Corollary 4.16]{MS}) to give a proof of
the desired uniform estimate for $I(x,t)$ for $d\ge 4$.
We split
the radial integration via $\supp(\rho) \subseteq E^1 \cup E^2$
where 
\begin{align*}
    E^1 &:= \{ 1/2 \leq r \leq 2 : |r + x_1/t| \le 1/|t| \}, \\
    E^2 &:= \{ 1/2 \leq r \leq 2 : |r + x_1/t| \ge 1/|t| \}
\end{align*}
and decompose $I(x,t)$ accordingly.
On the one hand, we have the uniform estimate
$|E^1| \le C|t|^{-1}$ for the measure of $E^1$.
On the other hand, 
$$
\int_{E^2} |t r^2 + x_1 r|^{-(d-1)/2} dr \le C_d \, |t|^{-(d-1)/2}
\int_{E^2} |r + (x_1/t)|^{-(d-1)/2} dr \le C |t|^{-1}
$$
when $d\ge 4$,
giving us the uniform estimate $|I(x,t)| \le C |t|^{-1}$
in that case. Note that the definition of $E^2$ keeps the integrand away from the singularity.

{\bf Preliminaries for $d=2,3$}. For the lower dimensional cases, we adjust the argument in the following way.
We proceed as before, splitting the radial integral into
two parts via a decomposition $F^1$ and $F^2$ where $F^1$ contains
 $E^1$ (and hence $F^2$ will be
contained in $E^2$) but will nevertheless still satisfy the uniform bound $|F^1| \le C |t|^{-1}$.
Thus we can estimate the part of $I(x,t)$ defined over $F^1$ by a constant multiple of
$|t|^{-1}$ as before -- this
only uses the trivial bound $|{\widehat{\sigma}}(y(r))| \lesssim 1$\footnote{We write $A\lesssim B$ if
there exists a constant $c>0$ such that $A \le c B$. Also we write $A \sim B$ if both $A\lesssim B$
and $B \lesssim A$ hold.}
and the estimate $|F^1| \lesssim |t|^{-1}$.
For the part over~$F^2$,
we need the more refined asymptotic expansion
$$
{\widehat{\sigma}}(y(r)) \ = \ c_1 \frac{e^{i |y(r)|}}{|y(r)|^{(d-1)/2}}
+ c_2 \frac{e^{-i |y(r)|}}{|y(r)|^{(d-1)/2}} + {\mathcal E}(y(r)),
$$
(see, for example, \cite[Example 10.4.3]{Graf})
valid for large $|y(r)|$,
where $c_1$ and $c_2$ are constants depending only on $d$,
and the uniform estimate $|{\mathcal E}(y(r))| \le C_d |y(r)|^{-(d+1)/2}$
holds. The argument above works as before for the integral
of ${\mathcal E}(y(r))$ over~$F^2$ (since $F^2$ is contained
in $E^2$) but is now valid for $d=2$ and $d=3$.

Hence matters are reduced to bounding integrals of the form
\begin{equation}\label{osciint}
\int_{F^2} e^{\pm i\phi(r)} \rho(r) \, {\phi(r)}^{-(d-1)/2} dr
\end{equation}
where $\phi(r) := |y(r)|$.
We will need to examine the phase $\phi(r)$ more closely;
write $\phi(r) = \sqrt{f(r)}$ where
$f(r) := (t r^2 + x_1 r)^2 + r^2 |x'|^2$ and
$x' = (x_2,\ldots, x_d)$ so that $x = (x_1, x')$.

{\bf Dimension $d=3$}. We may assume that $|x_1| \sim |t|$ and $|x'|/|t| \le c_0$
for any small absolute constant $c_0$ (to be determined
later); otherwise, $\phi(r) \gtrsim |t|$
on the support of $\rho$. Then, no matter how we define $F^1$ and $F^2$ (in this case,
we can take $F^1 = E^1$ and $F^2 = E^2$), we obtain
$$
\Bigl|\int_{F^2} e^{\pm i\phi(r)} \rho(r) \, {\phi(r)}^{-1} dr \Bigr| \ \lesssim \
\int_{F^2} |\rho(r)|\,|\phi(r)|^{-1} dr \ \lesssim \ |t|^{-1}.
$$
Now, when $d=3$, we can write the integral
in~\eqref{osciint} as
\begin{equation}\label{osciint1}
\mp 2i\int_{F^2} \frac{d}{dr}\bigl[ e^{\pm i\phi(r)} \bigr] \rho(r)
f'(r)^{-1} dr.    
\end{equation}
Although the derivative $f'$ can vanish to order two on the support of $\rho$, this cannot happen when $|x_1| \sim |t|$ and $|x'|/|t|$ is small. In this case, we
enlarge the set $E^1$
to  $F^1 := E^1 \cup \{ 1/2 \leq r \leq 2: |f'(r)| \le |t| \}$ and so
$$
F^2 \ := \ \{ 1/2 \leq r \leq 2 : |r + (x_1/t)|\ge 1/|t| \ {\rm and} \
|f'(r)|\ge |t| \}
$$
which is contained in $E^2$. We will see that the
uniform estimate $|F^1| \lesssim |t|^{-1}$ still holds.
On the other hand,  the set $F^2$ can be written as a finite union of intervals such that $\rho / f'$ is monotone on each subinterval. Taking one such interval $(a,b)$, we then integrate by parts:
\begin{align}\label{polkadot}
&\ \Big|\int_{a}^b \frac{d}{dr} \bigl[e^{\pm i\phi(r)} \bigr] \rho(r)
f'(r)^{-1} dr\Big|\\\nonumber
\le &\ \Big|\Big[e^{\pm i\phi(r)}\rho(r)
f'(r)^{-1}\Big]_a^b\Big|+\int_{a}^b \Big|e^{\pm i\phi(r)} \Big(\rho(r)
f'(r)^{-1}\Big)'\Big| dr\\\nonumber
\lesssim&\ \big[\min_{r\in (a,b)}|f'(r)|\big]^{-1}+\Big|\int_{a}^b \Big(\rho(r)
f'(r)^{-1}\Big)' dr\Big|\\\nonumber
\lesssim&\ \big[\min_{r\in (a,b)}|f'(r)|\big]^{-1} \lesssim
\ |t|^{-1},\end{align}
which yields the desired bound.
To prove the bound on $F^1$, a computation shows that $f'$
has two real roots away from $r=0$;
$f'(r) =  4r t^2 (r - r_{+})(r - r_{-})$
where
$$
r_{\pm} \ := \ - \frac{3}{4}\frac{x_1}{t} \ \pm \ \frac{1}{4}
\sqrt{\frac{x_1^2}{t^2} - 8\frac{|x'|^2}{t^2}}.
$$
The constant $c_0 >0$ above is chosen small enough to guarantee
that $r_{\pm}$ are real and $|r_{+} - r_{-}| \sim 1$.
From these observations we find that either $|r(r - r_{+})|\gtrsim 1$ or $|r(r - r_{-})|\gtrsim 1$ so that $|f'(r)|\le |t|$ forces either $|r - r_+| \lesssim |t|^{-1}$ or $|r - r_-| \lesssim |t|^{-1}$ and therefore $|F^1| \lesssim |t|^{-1}$ (if the roots $r_{+}$ and
$r_{-}$ were not separated, we would be stuck with
the bad bound $|F^1| \lesssim |t|^{-1/2}$).
This finishes the case $d=3$.

For the case $d=2$, observe that the above arguments break down 
and in particular the argument reducing to the case where $|x'|/|t|$ is small is not valid when $d=2$.  

{\bf Dimension $d=2$}. It remains to establish the uniform bound $|I(x,t)| \lesssim |t|^{-5/6}$. To do this, we go back to~\eqref{osciint} with the choice $F^1 = E^1$ and $F^2 = E^2$ and write the integral
appearing there as
\begin{equation}\label{osciint2}
\int_{E^2} e^{\pm i\phi(r)} \rho(r) \, {\phi(r)}^{-(d-1)/2} dr
\ = \ \frac{1}{\sqrt{t}} \int_{E^2} e^{\pm i{t\sqrt{g(r)}}} \rho(r) \, {g(r)}^{-1/4} dr    
\end{equation}
where $g(r) := [\phi(r)/t]^2$. Hence  matters are reduced to showing that
\begin{equation}\label{osciint3}
I \ := \  \int_{E^2} e^{\pm i{t\sqrt{g(r)}}} \rho(r) \, {g(r)}^{-1/4} dr
\ = \ O(|t|^{-1/3})    
\end{equation}
holds uniformly in $x$. We begin by reducing the parameters to the region $|x_1| \sim |t|$
and $|x_2| \lesssim |t|$. Recall
that $g(r)/r^2 = [r + x_1/t]^2 + x_2^2/t^2$ and so
\begin{equation}\label{g'}
\frac{g'(r)}{2r} \ = \ \big(r + \frac{x_1}{t}\big) \big(2r + \frac{x_1}{t}\big) \ + \
\frac{x_2^2}{t^2}.
\end{equation}
Hence if $|x_1| \ll |t|$, we have $g'(r) \sim 2r^2 + x_2^2/t^2 \gtrsim 1$
and $g(r) \sim r^2 + x_2^2/t^2$ on the support of $\rho$. We
write
\begin{equation}\label{osciint4}
I=\pm\frac{2}{it}\int_{E^2}  \frac{d}{dr}\bigl[e^{\pm i{t\sqrt{g(r)}}}\bigr] \rho(r) \, \frac{{g(r)}^{1/4}}{g'(r)} dr,
\end{equation}
and breaking up $E^2$ into intervals where ${g}^{1/4}/g'$ is monotonic as before, this yields the bound $I = O(|t|^{-1})$ via integration by parts, noting that $|{g}^{1/4}/g'|\lesssim 1.$

 Next, if $|x_1| \gg |t|$, we have
$g'(r), g(r) \sim (x_1^2 + x_2^2)/t^2 \gtrsim 1$ which again
leads to the bound $I = O(|t|^{-1})$. Finally if $|x_1| \sim |t|$
and $|x_2| \gg |t|$, we have $g'(r), g(r) \sim x_2^2/t^2 \gtrsim 1$, again
leading to the bound $I = O(|t|^{-1})$,  all of which are better than
the claimed bound for $I$ in \eqref{osciint3}.

Therefore we may assume that $|x_1| \sim |t|$ and $|x_2| \lesssim |t|$
(recall that when $d\ge 3$, the
reduction to $|x'| \gtrsim |t|$ was straightforward -- this is
not the case when $d=2$).
Recall that $E^2 := \{1/2 \leq r \leq 2 : |r + x_1/t| \ge 1/|t| \}$. 
If $\Phi(r) := \sqrt{g(r)}$, then $\Phi'(r) = g'(r)/2\sqrt{g(r)}$ and
$$
2 \Phi''(r) \ = \ \frac{g''(r) g(r) + [g'(r)]^2/2}{g(r)^{3/2}}
$$
where $g'(r)$ is given in \eqref{g'} and 
\begin{equation}\label{roots-2nd}
g''(r)/2 = r^2 + x_2^2/t^2 + 5 (r + x_1/t) (r + x_1/5t) \ = \ 6 (r - r_{+}) ( r - r_{-}).
\end{equation}
Here the roots $r_{\pm}$ may be complex but when $|x_2| \ll |t|$, the roots are real and 

\begin{equation*}
   \Big|r_{\pm}  -\Big( -\frac{x_1}{2t} \pm \frac{1}{\sqrt{3}}\cdot\frac{|x_1|}{2|t|}\Big)\Big| \ll 1. 
\end{equation*}

We split $E^2 = E^2_1 \cup E^2_2$ where 
\begin{equation*}
    E^2_1 := \{ 1/2 \leq r  \leq 2 : 1/|t| \le | r + x_1/t| \le \epsilon_0\}
\end{equation*}
for some small $\epsilon_0>0$ and decompose $I = I_1 + I_2$ accordingly.
Hence on $E^2_1$, $g''(r) \sim 1$ if $\epsilon_0$ is suffficently small (less than $1/100$ is enough). 
To analyse $I_1$ we consider two cases: $|x_2| \sim |t|$ and $|x_2| \ll |t|$.  In the first case, provided $\epsilon_0 > 0$ is sufficiently small, both $g(r), g'(r) \sim 1$
and so an integration by parts argument, as shown in \eqref{osciint4}, shows $|I_1| \lesssim |t|^{-1}$ as before. 
In the second case,  fix a small positive (absolute) 
$\delta>0$ such that $|x_2| \le \delta |t|$. In this case, we
split $E^2_1 = E^2_{1,1} \cup E^2_{1,2}$ further where $E^2_{1,1} := \{ r \in E^2_1 :
|r + x_1/t|^2 \le \epsilon x_2^2/t^2 \}$ for some small $\epsilon > 0$ and write $I_1 = I_{1,1} + I_{1,2}$
accordingly. 

On $E^2_{1,1}$, 
we have $g(r) \sim x_2^2/t^2$ and, by the definition of $E_1^2$, $$
|g'(r)| \le 10 |r + x_1/t| + x_2^2/t^2.
$$
The latter implies that $g'(r)^2 \le 10^{-2} x_2^2/t^2$
if $\epsilon, \delta >0$ are small enough. Therefore 
$$
|\Phi''(r) | \gtrsim \ 1/\sqrt{g(r)} \ \sim \ |t|/|x_2|
$$ 
on $E^2_{1,1}$. Since the amplitude $\rho(r)/g(r)^{1/4}$ in \eqref{osciint2} is $ \sim |t|^{1/2}/|x_2|^{1/2}$ on $E^2_{1,1}$,
an application of van der Corput's lemma, together with
an integration by parts argument (see for example \cite[pp. 234]{Stein}),  gives the bound
$|I_{1,1}| \lesssim |t|^{-1/2}$.

On $E^2_{1,2}$, $g(r) \sim |r + x_1/t|^2$ and 
$$
|g'(r)| \ \gtrsim \ |r+ x_1/t| - 10 \epsilon^{-1} |r + x_1/t|^2 \gtrsim \ |r + x_1/t|
$$
if we choose $\epsilon_0 \le 10^{-2} \epsilon$. Hence by integrating by parts, as in \eqref{osciint4}, we have
$$
|I_{1,2}|  \ \lesssim \ |t|^{-1} \max_{r \in E^2} |r + x_1/t|^{-1/2} \ \lesssim \ |t|^{-1/2}
$$
which, together with our bound on $I_{1,1}$,  implies $|I_1| \lesssim |t|^{-1/2}$ in this second case.
Hence in either case, we have $I_1 = O(|t|^{-1/2})$. 

To estimate $I_2$, we note that $g(r) \sim 1$ on $E^2_2$. 
By splitting the set $E^2_2$ into two parts where $g'(r)$
is respectively small and large, one can easily take
care of the part where $g'(r)$ is bounded away from zero
by integrating by parts (recall now $g(r) \sim 1$) as in \eqref{osciint4}, obtaining
a contribution of $O(|t|^{-1})$ for this part of the  integral $I_2$.
Hence we may assume that on $E^2_2$, we have the further restriction
that $|g'(r)| \le \eta$ for some small absolute constant $\eta>0$.

To understand
the size of $g''(r)$, we
split $E^2_2 = E^2_{2,1} \cup E^2_{2,2}$ where
$$E^2_{2,1} := \{ r \in E^2_2 : |r + x_1/2t| \le \epsilon_0\}$$ and
write $I_2 = I_{2,1} + I_{2,2}$ accordingly. On $E^2_{2,1}$, since  $|g'(r)| \le \eta$, we see from \eqref{g'} that
$|x_2|^2/t^2 \le 10(\epsilon_0 + \eta)$. Hence by \eqref{roots-2nd}, we see that $|g''(r)| \sim 1$ on $E^2_{2,1}$.
Therefore, since $g \sim 1$ on $E^2_2$, we see that $|\Phi''(r)| \gtrsim 1$ on $E^2_{2,1}$ and now an application of van der Corput's lemma as before  shows that $|I_{2,1}| \lesssim |t|^{-1/2}$.

Finally we turn to $I_{2,2}$ where we are integrating
over $E^2_{2,2}$ and in particular we have $|g'''(r)| \ge 24 \epsilon_0$ since
$g'''(r) = 24 (r + x_1/2t)$. We also have
$|g'(r)| \le \eta$ and $g(r) \sim 1$ for $r\in E^2_2$ and so we compute
$$
\frac{d^3}{dr^3}\sqrt{g(r)} \ = \ 
\frac{g'''(r)g(r)^2 - (3/2)g'(r)[g''(r)g(r) - (1/2)g'(r)^2]}{2g(r)^{5/2}}
$$
and deduce the bound $|d^3/dr^3 \sqrt{g(r)}| \gtrsim 1$ on $E^2_{2,2}$
if $\epsilon_0$ and $\eta$ are chosen so that $\eta \ll \epsilon_0$. 
Hence a final application of van der Corput's lemma, together with
an integration by parts argument,  gives the bound $|I_{2,2}| \lesssim |t|^{-1/3}$
which implies that $|I_2| \lesssim |t|^{-1/3}$ and this completes the proof of the claim $I = O(|t|^{-1/3})$.
\end{proof}

\section{Proof of Theorem~\ref{main}}

A dispersive estimate of the following type with $d=1$ was proven by Ponce and Vega~\cite[Lemma 2.4]{PV} with decay of the order $|t|^{-(1/2-1/p)}$.  Thus one could argue that we have more dispersion in the higher dimensional cases, however the decay rate does not continue to increase with the dimension. This is somehow consistent with the notion that the higher dimensional equation continues to model unidimensional behaviour. 

\begin{prop}\label{disp}
Let $d\ge 3$ and $2\le p<\infty$. Then there is a constant $C\equiv C(d,p)$ such that
$$
\big\|e^{t\mathcal{R}_1\Delta} f\big\|_{L^p(\R^{d})}\le C|t|^{-2(\frac{1}{2}-\frac{1}{p})}\big\|(-\Delta)^{(d-2)(\frac{1}{2}-\frac{1}{p})} f\big\|_{L^{p'}\!(\R^d)}.
$$
Moreover, with $d=2$, 
$$
\big\|e^{t\mathcal{R}_1\Delta} f\big\|_{L^p(\R^{2})}\le C|t|^{-\frac{5}{3}(\frac{1}{2}-\frac{1}{p})}\big\|(-\Delta)^{(d-\frac{5}{3})(\frac{1}{2}-\frac{1}{p})} f\big\|_{L^{p'}\!(\R^2)}.
$$
Both rates of decay are optimal.
\end{prop}

\begin{proof} The loss of Sobolev regularity is necessary given the decay rates, by scaling.  The decay rates cannot be improved as otherwise the estimates of Theorem~\ref{main} would hold in larger ranges, which is not possible. 

To prove the estimates, we consider $\psi:\mathbb{R}\to \mathbb{R}$ a  Schwartz function supported in~$[1/2,2]$ for which $\sum_{k\in\mathbb{Z}} \psi(2^{-k} |\cdot|)=1$. Defining the projection operators~$P_k$  by 
\begin{equation}\label{littp}
  P_k f := \big(\psi(2^{-k} |\cdot|)\widehat{f}\,\big)^\vee,  
\end{equation}
the Littlewood--Paley inequality yields
\begin{align*}
\| e^{t\mathcal{R}_1\Delta} f\|_{L^p(\R^{d})}&\le C \Big\| \Big(\sum_{k\in\mathbb{Z}} |e^{t\mathcal{R}_1\Delta}P_k f|^2\Big)^{1/2}\Big\|_{L^p(\R^{d})}\\
&\le C \Big(\sum_{k\in\mathbb{Z}} \big\|e^{t\mathcal{R}_1\Delta}P_k f\big\|^2_{L^p(\R^{d})}\Big)^{1/2}, 
\end{align*}
where in the final estimate we use the triangle inequality in $L^{p/2}$. Thus, with $d\ge3$, it will suffice to prove that, for all $k\in \mathbb{Z}$,
\begin{equation}\label{this}
\| e^{t\mathcal{R}_1\Delta}P_k f\|_{L^p(\R^{d})}\le C|t|^{-2(\frac{1}{2}-\frac{1}{p})}2^{(d-2)(1-\frac{2}{p})k}\|f\|_{L^{p'}\!(\R^d)}.
\end{equation}
Then, with $\widetilde{P}_k$ defined in a similar way as $P_k$ but with the cut-off function equal to one on the support of $\psi$, we would obtain
\begin{align*}
\| e^{t\mathcal{R}_1\Delta} f\|_{L^p(\R^{d})}&\le C|t|^{-2(\frac{1}{2}-\frac{1}{p})}\Big(\sum_{k\in\mathbb{Z}} 2^{(d-2)(1-\frac{2}{p})k}\|\widetilde{P}_kf\|^2_{L^{p'}\!(\R^d)}\Big)^{1/2}\\
&\le C|t|^{-2(\frac{1}{2}-\frac{1}{p})}\Big\|\Big(\sum_{k\in\mathbb{Z}} 2^{(d-2)(1-\frac{2}{p})k}|\widetilde{P}_kf|^2\Big)^{1/2}\Big\|_{L^{p'}\!(\R^d)},
\end{align*}
and the desired inequality would follow by a further Littlewood--Paley application.

By scaling invariance, it will suffice to prove~\eqref{this} with $k=0$. That is to say, an estimate for the operator $T$ defined by 
\begin{align*}
Tf(x,t)&:= \int_{\R^d} \psi(|\xi|)\widehat{f} (\xi)\, e^{i[t\xi_1|\xi| +x\cdot\xi]}d\xi
\end{align*}
with data $f\in L^{p'}\!(\R^d)$.  
By Plancherel's theorem, we have that 
\begin{equation}\label{l2}\|Tf(\cdot,t)\|_2\le C\,\|f\|_2.\end{equation} 
On the other hand, recalling that the Fourier transform is defined by $$\widehat{f}(\xi):=\int_{\R^d} f(y)\, e^{-iy\cdot\xi}\, dy,$$ by Fubini's theorem, we can write
\begin{equation}\label{bla}
Tf(x,t)=\int_{\R^d} f(y) \int_{\R^d} \psi(|\xi|)\, e^{i[t\xi_1|\xi| +(x-y)\cdot\xi]}d\xi \,dy.
\end{equation}
Applying Lemma~\ref{thelemma}, we obtain 
$$
\Big|\int_{\R^d} \psi(|\xi|) e^{i[t\xi_1|\xi| +(x-y)\cdot\xi]}d\xi\Big| \le C|t|^{-1}.
$$
Substituting into~\eqref{bla}, this yields a dispersive estimate, 
\begin{equation*}
\|Tf(\cdot,t)\|_{L^\infty(\R^d)}\le C|t|^{-1}\|f\|_{L^1(\R^d)}, 
\end{equation*}
so that by Riesz--Thorin interpolation with our $L^2$ estimate~\eqref{l2}, we obtain
\begin{equation}\label{dis}
\|Tf(\cdot,t)\|_{L^p(\R^d)}\le \frac{C}{|t|^{1-2/p}}\|f\|_{L^{p'}\!(\R^d)},\qquad p\ge 2, 
\end{equation}
which is~\eqref{this} with $k=0$, and so we are done.

When $d=2$, the same argument applies. In that case, the estimate~\eqref{dis} takes the form
$$
\|Tf(\cdot,t)\|_{L^p(\R^2)}\le \frac{C}{|t|^{\frac{5}{6}(1-2/p)}}\|f\|_{L^{p'}\!(\R^2)},\qquad p\ge 2, 
$$
and so~\eqref{this} is changed accordingly.
\end{proof}

We now continue with the proof of the main Theorem~\ref{main}.

\begin{proof}[Proof of Theorem~\ref{main}]
First we calculate the necessary condition when $d\ge 3$. By the change of variables $\xi\to R\xi$, $x\to R^{-1} x$ and $t\to R^{-2} t$, then letting $R$ tend to zero and infinity, one can calculate that it is necessary that $s=d(\frac{1}{2}-\frac{1}{q})-\frac{2}{r}$ for such an estimate to hold.  
Now define $f$ by $$\widehat{f}(\xi):=\chi_{[0,R^{-1}]}(\xi_1)\chi_{[0,1]}(\xi_2)\ldots\chi_{[0,1]}(\xi_d),$$ so that 
\begin{equation}\label{one}
\| f\|_{\dot{H}^{s}}\le cR^{-1/2}. 
\end{equation}
On the other hand,
$$
e^{t\mathcal{R}_1\Delta} f(x)=\frac{1}{(2\pi)^d} \int_{0}^{R^{-1}}\!\!\! \int_{0}^1\ldots \int_{0}^1 e^{it|\xi|\xi_1}e^{i[x_1\xi_1+\ldots + x_d\xi_d]}\,d\xi_1\ldots d\xi_d.
$$
We see that as long as $|t|,|x_1|\le \frac{R}{2d}$ and $|x_2|,\ldots,|x_d|\le \frac{1}{2d}$, then 
$$
|e^{t\mathcal{R}_1\Delta} f(x)|\ge \frac{1}{(2\pi)^d} \int_{0}^{R^{-1}}\!\!\! \int_{0}^1\ldots \int_{0}^1 \cos (1)\, d\xi_1\ldots d\xi_d\ge c R^{-1}.
$$
Thus, 
\begin{equation}\label{two}
\| e^{t\mathcal{R}_1\Delta} f\|_{L_t^{r}(\R,L_x^q(\R^{d}))}\ge c R^{1/q+1/r-1}
\end{equation}
and so by comparing~\eqref{one} and~\eqref{two} and letting $R$ tend to infinity, we see that the
condition $\frac{2}{q} + \frac{2}{r} \le 1$ is necessary.

 With $d=2$, the condition $\frac{10}{q} + \frac{12}{r} \leq 5$ is deduced via an example of a different nature. It will be useful to let $\Sigma$ denote the phase function $\Sigma \colon \xi \mapsto \xi_1|\xi|$. This map has an associated hessian given by
\begin{equation*}
\partial_{\xi\xi}^2\Sigma(\xi) = \frac{1}{|\xi|^3}
\begin{bmatrix}
2 \xi_1^3 + 3 \xi_1 \xi_2^2 & \xi_2^3 \\
\xi_2^3 & \xi_1^3
\end{bmatrix}
\end{equation*}
and, moreover,
\begin{equation*}
    \det \partial_{\xi\xi}^2\Sigma(\xi) = \frac{2 \xi_1^2 - \xi_2^2}{|\xi|^2}.
\end{equation*}
Fix $\delta >0$ and some $\zeta \in \R^2$ bounded away from 0 satisfying $2 \zeta_1^2 = \zeta_2^2$ so that $\det\partial_{\xi\xi}^2\Sigma(\zeta) = 0$. Define the matrices
\begin{equation*}
    A := \frac{1}{(\zeta_1^6 + \zeta_2^6)^{1/2}}
    \begin{bmatrix}
    \zeta_2^3 & -\zeta_1^3 \\
    \zeta_1^3 & \zeta_2^3
    \end{bmatrix}, \qquad   D_{\delta} := \begin{bmatrix}
    \delta^{1/2} & 0 \\
    0 & \delta^{1/3}
    \end{bmatrix}.
\end{equation*}
The columns of $A$ form an orthonormal basis of eigenvectors for $\partial_{\xi\xi}^2\Sigma(\zeta)$ with corresponding eigenvalues $\lambda := 3\zeta_1(\zeta_1^2 + \zeta_2^2)$ and $0$. Thus, $A^{\!\top}\partial_{\xi\xi}^2\Sigma(\zeta)A$ is a rank 1 diagonal matrix with top-left entry equal to $\lambda$. 

Take $\vartheta$ to be a non-negative Schwartz function on $\R^2$ with support in the unit ball and $\vartheta(0) \neq 0$. Let $f$ be defined by 
\begin{equation*}
    \widehat{f}(\xi) := \vartheta \circ D_{\delta}^{-1}A^{-1}(\xi - \zeta)
\end{equation*} 
so that 
\begin{equation}\label{2d sharp 1}
    \|f\|_{\dot{H}^{s}} \leq c \delta^{5/12}.
\end{equation}
By a change of variables,
\begin{align*}
    e^{t\mathcal{R}_1\Delta} f(x) &= \frac{\delta^{5/6}}{(2\pi)^2} \int_{\R^2} \vartheta(\xi)e^{it\Sigma(\zeta + A D_{\delta}\xi)} e^{ix \cdot (\zeta + A D_{\delta}\xi)} d \xi \\
    &= \frac{\delta^{5/6}}{(2\pi)^2}e^{it\Sigma(\zeta)} e^{ix \cdot \zeta} \int_{\R^2} \vartheta(\xi)e^{i t (\lambda \delta\xi_1^2/2 + E_{\delta}(\xi))} e^{iD_{\delta}A^{\!\top}(x+\partial_{\xi}\Sigma(\zeta)) \cdot \xi} d \xi
\end{align*}
where, by Taylor's theorem,
\begin{equation*}
    E_{\delta}(\xi) := \Sigma(\zeta + A D_{\delta}\xi) - \Sigma(\zeta) - \partial_{\xi}\Sigma(\zeta)\cdot A D_{\delta}\xi - \tfrac{1}{2}\partial_{\xi\xi}^2\Sigma(\zeta) A D_{\delta}\xi \cdot A D_{\delta}\xi
\end{equation*}
satisfies $|E_{\delta}(\xi)| = O(\delta)$ on the support of $\vartheta$. Here we have used the properties of $A$ to write the final term in the definition of $E_{\delta}(\xi)$ as $\lambda \delta \xi_1^2/2$. If $c_0 > 0$ is a suitably small (universal) constant, then it follows that $|e^{t\mathcal{R}_1\Delta} f(x)| \gtrsim \delta^{5/6}$ whenever 
\begin{equation*}
    |t| \leq c_0 \delta^{-1} \quad \textrm{and} \quad |D_{\delta}A^{\!\top}(x+\partial_{\xi}\Sigma(\zeta))| \leq c_0.
\end{equation*}
Consequently,
\begin{equation}\label{2d sharp 2}
    \|e^{t\mathcal{R}_1\Delta} f\|_{L_t^{r}(\R,L_x^q(\R^{2}))} \geq c \delta^{5/6(1-1/q) - 1/r}
\end{equation}
and the remaining necessary condition follows by comparing \eqref{2d sharp 1} and \eqref{2d sharp 2} and letting $\delta$ tend to zero.

To prove the estimates,  the argument is completed using a variant of the $TT^*$ argument of Tomas~\cite{T}.  We consider first the estimates on the sharp line $\frac{2}{q}+\frac{2}{r}=1$ for $d \geq 3$. Note that the $r=\infty$ estimate is a consequence of~\eqref{l2}, so we can suppose $r$ is finite.
By duality it will suffice to prove
$$
\Big\|(-\Delta)^{-s/2}\int e^{-t\mathcal{R}_1\Delta} F(\cdot,t)\,dt \Big\|_{L_x^2(\R^d)}\le
C\,\|F\|_{L_t^{r'}L_x^{q'}},
$$
where $s=d(\frac{1}{2}-\frac{1}{q})-\frac{2}{r}$, which  is equivalent to
\begin{equation*}\label{bil} \int\!\!\int\Big\langle
(-\Delta)^{-s}e^{-t\mathcal{R}_1\Delta} F(\cdot,\tau),e^{-\tau\mathcal{R}_1\Delta}
G(\cdot,t)\Big\rangle_{\!x}dtd\tau\le
C\,\|F\|_{L_t^{r'}L_x^{q'}}\|G\|_{L_t^{r'}L_x^{q'}}.
\end{equation*}
This can be rewritten as
\begin{equation*}\label{TTstar}
 \int  \Big\langle
\int (-\Delta)^{-s}e^{-(t-\tau)\mathcal{R}_1\Delta} F(\cdot,\tau)\,d\tau,
G(\cdot,t)\Big\rangle_{\!x}\,dt\le C\,\|F\|_{L_t^{r'}L_x^{q'}}\|G\|_{L_t^{r'}L_x^{q'}}, 
\end{equation*}
which would follow, by H\"older's inequality, from
\begin{equation*}\label{TTstar2}
\Big\|\int (-\Delta)^{-s}e^{-(t-\tau)\mathcal{R}_1\Delta} F(\cdot,\tau)\,d\tau\Big\|_{L_t^{r}L_x^{q}}\le
C\,\|F\|_{L_t^{r'}L_x^{q'}}\end{equation*}
Noting that on the sharp line $s=(d-2)(\frac{1}{2}-\frac{1}{q})$, this is a consequence of Minkowski's integral inequality, the decay estimate of Proposition~\ref{disp},  and the one dimensional Hardy--Littlewood--Sobolev inequality (see for example \cite[Proposition 7.8]{MS});
\begin{align*}
\Big\|\int (-\Delta)^{-s}e^{-(t-\tau)\mathcal{R}_1\Delta} F(\cdot,\tau)\,d\tau\Big\|_{L_t^{r}L_x^{q}}&\le \Big\|\int \big\|(-\Delta)^{-s}e^{-(t-\tau)\mathcal{R}_1\Delta} F(\cdot,\tau)\big\|_{L_x^{q}}\,d\tau\Big\|_{L^r_t}\\
&\le \Big\|\int \frac{\|F(\cdot,\tau)\|_{L_x^{q'}}}{|t-\tau|^{1-\frac{2}{q}}}\,d\tau\Big\|_{L^r_t}\\&
\le C\,\|F\|_{L_t^{r'}L_x^{q'}}.
\end{align*}
For the application of the Hardy--Littlewood--Sobolev inequality we require that $0<1-\frac{2}{q}=\frac{2}{r}<1$, which follows from the fact that $q$ and $r$ are finite. 

Finally, in order to take $q$ larger than those on the sharp line $\frac{1}{2}=\frac{1}{q}+\frac{1}{r}$,  we can apply the Hardy--Littlewood--Sobolev inequality (see \cite[Corollary 7.9]{MS}), in the form 
$$
\|e^{t\mathcal{R}_1\Delta} f\|_{L^{\tilde{q}}(\R^d)}\le C\,\|(-\Delta)^{s/2}e^{t\mathcal{R}_1\Delta} f\|_{L^{q}(\R^d)},\quad s=d\big(\tfrac{1}{q}-\tfrac{1}{\tilde{q}}\big),
$$
and so the proof for $d\ge3$ is complete. 

When $d=2$, the only thing that changes from above is that we initially consider estimates on the line $\frac{10}{q}+\frac{12}{r}=5$, so that $s=d(\frac{1}{2}-\frac{1}{q})-\frac{2}{r}$ becomes $$s=(d-5/3)\big(\tfrac{1}{2}-\tfrac{1}{q}\big)$$and the condition $0<1-\frac{2}{q}=\frac{2}{r}<1$ becomes $0<\frac{5}{6}-\frac{10}{6q}=\frac{2}{r}<1$.
This completes the proof.
\end{proof}
 

\section{Proof of Theorem~\ref{imprwellpos}}

After using our linear estimates to prove some {\it a priori }nonlinear estimates, we present the proof of Theorem~\ref{imprwellpos} in the final three subsections. In the first of these three parts we prove uniqueness, using the fact that $u\in L^1([0,T),W^{1,\infty}(\mathbb{R}^d))$. Then, we prove existence using a compactness argument. Finally we deduce continuous dependence.


\subsection{Linear estimates}

By an application of Sobolev embedding in the spatial variable and an application of H\"older's inequality in the time variable, our Strichartz estimates of Theorems~\ref{main} yield the following corollary. 

\begin{cor}\label{liestCoro}
Let $s>s_d-3/2$ where $s_d := d/2+1/2$ for $d\ge 3$ and $s_2 := 5/3$.  Then 
\begin{equation*}
    \Big(\int_0^1\left\|e^{t\mathcal{R}_1 \Delta}f\right\|^2_{L^{\infty}}dt\Big)^{1/2}\le C_{s}\left\|f\right\|_{H^s}.
\end{equation*}
\end{cor}

This will be a key ingredient in the proof of the following inhomogeneous Strichartz estimate.

\begin{lemma} \label{refinStri}
Let $s>s_d-1$ where $s_d := d/2+1/2$ for $d\ge 3$ and $s_2 := 5/3$.  Then 
\begin{equation*}
   \int_0^T \left\|w(\cdot,t)\right\|_{L^{\infty}} dt \le C_{s} T^{1/2} \left(\sup_{t\in[0,T]}\left\|w(\cdot,t)\right\|_{H^s}+\int_0^T \left\|F(\cdot,t) \right\|_{H^{s-1}} dt\right)
\end{equation*}
whenever $T\le 1$ and $w$ is a solution to
$ \partial_t w- \mathcal{R}_1 \Delta w= F.$
\end{lemma}

\begin{proof}
We again consider the smooth partition of unity $\sum_{k\in\mathbb{Z}} \psi(2^{-k} |\cdot|)=1$ and the Littlewood--Paley projection operators~$P_k$ defined by 
\begin{equation*}
  P_k f:= \big(\psi(2^{-k} |\cdot|)\widehat{f}\,\big)^\vee,\quad k\ge 2.   
\end{equation*}
This time we consider the lower frequencies together; $P_1f:=\sum_{k\le 1} \big(\psi(2^{-k} |\cdot|)\widehat{f}\,\big)^\vee$.  
By the triangle inequality and summing a geometric series, it will suffice to prove that, for all $k\ge 1$ and $s>s_d-1$,  
\begin{equation*}
   \int_0^T \left\|P_kw(\cdot,t)\right\|_{L^{\infty}} dt \le C_{s} T^{1/2} \left(\sup_{t\in[0,T]}\left\|P_kw(\cdot,t)\right\|_{H^s}+\int_0^T \left\|P_kF(\cdot,t) \right\|_{H^{s-1}} dt\right).
\end{equation*}

Following the arguments of~\cite{KenigKo}, we split the interval $[0,T]=\bigcup_m I_m$ into $2^{k}$ intervals $I_m=[a_m,b_m]$ with $(b_m-a_m) \sim 2^{-k}T$ and use the Cauchy--Schwarz inequality to estimate
\begin{equation}\label{pop}
   \int_0^T \left\| P_kw(\cdot,t)\right\|_{L^{\infty}} dt  \lesssim (2^{-k}T)^{1/2} \sum_{m} \Big(\int_{I_m} \left\|P_kw(\cdot,t)\right\|^2_{L^{\infty}} dt\Big)^{1/2}.\end{equation}
Now to estimate the right-hand side, we employ Duhamel's principle on each~$I_m$, so that
\begin{equation*}
    P_kw(\cdot,t)=e^{(t-a_m)\mathcal{R}_1 \Delta}P_kw(\cdot,a_m)+\int_{a_m}^t e^{(t-t')\mathcal{R}_1 \Delta} P_kF(\cdot,t')\, dt'
\end{equation*}
whenever $t\in I_m$.
Then, by an application of Minkowski's integral inequality in order to treat the second term as the first, and applications of Corollary~\ref{liestCoro}, we find that \eqref{pop} can be bounded by
\begin{equation*}
    \begin{aligned}
 &\lesssim (2^{-k}T)^{1/2}\sum_{m}  \left( \left\|P_kw(\cdot,a_m)\right\|_{H^{s-1/2}} + \int_{I_m}  \left\|P_kF(\cdot,t')\right\|_{H^{s-1/2}}\, dt'\right) \\
  &  \lesssim  T^{1/2}  \left( \sup_{t\in [0,T]} \left\|P_kw(\cdot,t)\right\|_{H^{s}}+\int_0^T \left\|P_kF(\cdot,t')\right\|_{H^{s-1}}\, dt'\right),
    \end{aligned}
\end{equation*}
which completes the proof.
\end{proof}


\subsection{Energy estimates}

We will need the Kato--Ponce commutator estimates.

\begin{lemma}\cite{FRAC, KP}\label{conmKP}
Let $s>0$ and write $[J^s,f]g:=J^s(fg)-fJ^sg$. Then
\begin{equation*}
    \big\|[J^s,f]g\big\|_{L^2} \lesssim   \left\|\nabla f\right\|_{L^{\infty}} \left\|J^{s-1}g\right\|_{L^2}+ \left\|J^s f\right\|_{L^2} \left\|g\right\|_{L^{\infty}}.
\end{equation*}
Moreover 
\begin{align*}
     \left\|J^s(fg)\right\|_{L^2} & \lesssim  \left\|J^s f\right\|_{L^{2}} \left\|g\right\|_{L^{\infty}}+ \left\|f\right\|_{L^{\infty}} \left\|J^s g\right\|_{L^{2}}. \label{eqfraLR}
\end{align*}
\end{lemma}

Using Lemma~\ref{conmKP}, we deduce the following {\it a priori} energy estimate for smooth solutions which we take in $H^{d+1}(\mathbb{R}^{d})$ from now on for definiteness.

\begin{lemma}\label{apriEST} Let $s\in(0, d+1]$. There is a constant~$c_{s}>0$ such that
\begin{equation*}
   \sup_{t\in[0,T]} \left\|u(t)\right\|_{H^{s}}^2 \le \left\|u_0\right\|_{H^s}^2+c_s \sup_{t\in[0,T]} \left\|u(t)\right\|_{H^s}^2  \int_{0}^T\left\|\nabla u(t)\right\|_{L^{\infty}}dt
\end{equation*}
whenever $u_0\in H^{d+1}(\mathbb{R}^{d})$  and $u\in C\big([0,T]; H^{d+1}(\mathbb{R}^{d})\big)$ solves \eqref{HBO-IVP}.
\end{lemma}
\begin{proof}
Applying $J^s$ to the higher dimensional Benjamin--Ono equation, multiplying by $J^s u$ and integrating in space yields
\begin{equation*}
    \frac{1}{2}\frac{d}{dt}\int_{\mathbb{R}^d} (J^s u)^2 \, dx=-\int_{\mathbb{R}^d} [J^s,u]\partial_{x_1} u J^s u \, dx-\int_{\mathbb{R}^d} u J^s \partial_{x_1} u J^s u \, dx.
\end{equation*}
The Riesz transform term  is easily seen to be zero using the fact that it is skew-self-adjoint. The first term on the right-hand side is bounded by the Cauchy--Schwarz inequality and the commutator estimate of  Lemma~\ref{conmKP}. Noting that $J^s\partial_{x_1}uJ^su=\frac{1}{2}\partial_{x_1}(J^su)^2$, the second
term is controlled by integrating by parts and using H\"older's inequality. Together we deduce
\begin{equation*}
    \frac{d}{dt}\left\|J^s u\right\|_{L^2}^2 \lesssim \left\|\nabla u \right\|_{L^{\infty}}\left\|J^s u\right\|_{L^2}^2.
\end{equation*}
Integrating in time, the desired inequality follows.
\end{proof}

\subsection{Nonlinear estimates}

As in the previous section we will prove {\it a priori} estimates for smooth solutions with respect to norms of lower regularity. The following estimate for the Lipschitz norm is a consequence of the inhomogeneous Strichartz estimate of Lemma~\ref{refinStri}.

\begin{lemma}\label{apriEstS}Let $s\in(s_d,d+1]$ where $s_d := d/2+1/2$ for $d\ge 3$ and $s_2 := 5/3$. For $T\le 1$, let
\begin{equation*}\label{kdef}
K(T):=\int_0^T \left\|u(t)\right\|_{L^\infty}+\left\|\nabla u(t)\right\|_{L^{\infty}} dt.
\end{equation*}
Then there is a constant~$C_{s}>0$ such that
\begin{equation*}
    K(T)\le C_sT^{1/2}\sup_{t\in[0,T]}\left\|u(t)\right\|_{H^s}\big(1+K(T)\big),
\end{equation*}
whenever $u_0\in H^{d+1}(\mathbb{R}^{d})$ and $u\in C\big([0,T]; H^{d+1}(\mathbb{R}^{d})\big)$ solves \eqref{HBO-IVP}.
\end{lemma}

\begin{proof}
   Applying Lemma~\ref{refinStri} with $w:=u$ and $F:=-u\partial_{x_1}u$ and also with $w:=\nabla u$ and $F := -\nabla(u\partial_{x_1}u)$, it follows that
\begin{equation}\label{liest8}
  K(T) \le C_{s} T^{1/2} \left(\sup_{t\in[0,T]}\left\|u(t)\right\|_{H^s}+\int_0^T \left\|(u\partial_{x_1}u)(t') \right\|_{H^{s-1}} dt'\right).
\end{equation}
Then using the second estimate of Lemma~\ref{conmKP}, 
\begin{equation*}\label{liest10}
    \begin{aligned}
   \left\|(u\partial_{x_1}u)(t) \right\|_{H^{s-1}}  &\lesssim \left\| u(t)\right\|_{L^{\infty}}\left\|J^{s-1}\partial_{x_1}u(t)\right\|_{L^{2}}+\left\|J^{s-1} u(t)\right\|_{L^2}\left\|\partial_{x_1}u(t)\right\|_{L^{\infty}} \\
    &\lesssim \big(\left\|u(t)\right\|_{L^{\infty}}+\left\|\nabla u(t)\right\|_{L^{\infty}}\big)\sup_{t\in[0,T]}\left\|u(t)\right\|_{H^s},
    \end{aligned}
\end{equation*}
which can be plugged into \eqref{liest8} to yield the desired inequality.
\end{proof}

With $s>d/2+1$, Theorem~\ref{imprwellpos} follows by a parabolic regularization of~\eqref{HBO-IVP};  an additional term $-\mu \Delta u$ is added, after which the limit $\mu \to 0$ is taken. The argument also yields a blow-up criteria. The result follows directly from the arguments of Iorio~\cite{Iorio} and so is omitted. 

\begin{lemma}\label{comwellp}
Let $s> d/2+1$. Then, for any $u_0\in H^s(\mathbb{R}^d)$, there exists a time $T=T(\left\|u_0\right\|_{H^s})$ and a unique maximal solution~$u$ to \eqref{HBO-IVP} in $C([0,T^\ast);H^s(\mathbb{R}^d))$.
Moreover, if the maximal time of existence $T^{\ast}$ is finite, then
\begin{equation*}
    \lim_{t \to T^{\ast}} \left\|u(t)\right\|_{H^s}=\infty,
\end{equation*}
and the flow map $u_0 \mapsto u(t)$ is continuous from $H^s(\mathbb{R}^d)$ to $H^s(\mathbb{R}^d)$.
\end{lemma}

Given this lemma, we first prove that the smooth solutions exist long enough for our purposes, taking advantage of the blow-up criteria. We then provide another nonlinear estimate. The proof  follows closely the arguments of~\cite{LinarFKP}.

\begin{lemma}\label{apriESTlower}
Let $s\in(s_d,d+1]$ where $s_d := d/2+1/2$ for $d\ge3$ and $s_2 := 5/3$.  Then there is a constant $A_s>0$ such that, for all $u_0\in H^{d+1}(\mathbb{R}^d)$,   there is a solution $u\in C([0,T^\ast); H^{d+1}(\mathbb{R}^d))$ of \eqref{HBO-IVP} with $T^\ast \ge (1+A_s\left\|u_0\right\|_{H^s})^{-2}$. Moreover, there is a constant $K_s>0$ such that 
\begin{equation*}
    \sup_{t\in[0,T]}\left\|u(t)\right\|_{H^s} \le 2 \left\|u_0\right\|_{H^s}, \quad\text{and}\quad K(T)\le K_s
\end{equation*}
whenever $T\le(1+A_s\left\|u_0\right\|_{H^s})^{-2}$.
\end{lemma}

\begin{proof} Set $A_s := 8(1+\max\{c_s,c_{d+1}\})C_s$ where $c_s$ and $c_{d+1}$ are the constants appearing in  Lemma~\ref{apriEST} and $C_s$ is the constant of  Lemma~\ref{apriEstS}. We consider any time $T\le(1+A_s\left\|u_0\right\|_{H^s})^{-2}$ for which
\begin{equation}\label{beg}
\sup_{t\in[0,T]}\left\|u(t)\right\|_{H^s}\le 2\left\|u_0\right\|_{H^s}
\end{equation}
for all $u_0\in H^{d+1}$. Then, by Lemma~\ref{apriEstS},
$$K(T)\le 2 C_s T^{1/2}\left\|u_0\right\|_{H^s}\big(1+K(T)\big).$$
Calculating we find
$$K(T)\le \frac{1}{3\max\{c_s,c_{d+1}\}}.$$
From the energy estimates of Lemma~\ref{apriEST} we deduce that both
\begin{equation*}
\sup_{t\in[0,T]}\left\|u(t)\right\|_{H^{s}}^2 \le \frac{3}{2} \left\|u_0\right\|_{H^{s}}^2
\quad\text{and}\quad
\sup_{t\in[0,T]}\left\|u(t)\right\|_{H^{d+1}}^2 \le \frac{3}{2} \left\|u_0\right\|_{H^{d+1}}^2
\end{equation*}
for all $u_0\in H^{d+1}$. In view of the blow-up criteria of Lemma~\ref{comwellp},  the latter estimate implies that we can take $T^{\ast}>T$. On the other hand, the former estimate and continuity implies that $T$ was not the largest time for which \eqref{beg} holds.
We conclude that the largest such $T$ must be as least as large as $(1+A_s\left\|u_0\right\|_{H^s})^{-2}$ which completes the proof. 
\end{proof}


\subsection{Uniqueness}\label{equniqu1}

Let $u_1$ and $u_2$ be two solutions of~\eqref{HBO-IVP}  in our class
$$
C\big([0,T);H^s(\mathbb{R}^d)\big)\cap L^1\big([0,T);W^{1,\infty}(\mathbb{R}^d)\big),
$$
with respective initial data $u_1(\cdot,0)=\phi_{1}$ and $u_2(\cdot,0)=\phi_2$.
By setting $v:=u_1-u_2$, we find that
\begin{equation*}
    \partial_t v-\mathcal{R}_1 \Delta v +u_1\partial_{x_1}u_1-u_2\partial_{x_1}u_2=0,
\end{equation*}
which can be rewritten as 
\begin{equation*}
    \partial_t v-\mathcal{R}_1 \Delta v +\tfrac{1}{2}\partial_{x_1}\big((u_1+u_2)v\big)=0.
\end{equation*}
Taking the inner product in $L^2(\mathbb{R}^d)$ with $v$, we arrive at
\begin{equation*}
  \frac{1}{2}\frac{d}{dt} \int_{\mathbb{R}^d} v^2 \, dx=-\frac{1}{4}\int_{\mathbb{R}^d} \partial_{x_1}(u_1+u_2)v^2 \, dx.   
\end{equation*}
Again by skew-adjointness, the Riesz transform term is zero, and one can arrive to the form of the right-hand side by integrating by parts twice. 
Thus,
\begin{equation*}
\frac{1}{2}\frac{d}{dt}\left\|v\right\|_{L^2}^2 \lesssim \big(\left\|\nabla u_1\right\|_{L^{\infty}}+\left\|\nabla u_2\right\|_{L^{\infty}}\big)\left\|v\right\|_{L^2}^2.
\end{equation*}
An application of Gronwall's inequality (see for example~\cite[Theorem 1.12]{taodisp}) gives
\begin{equation*}
    \sup_{t\in[0,T]}\left\|u_1(t)-u_2(t)\right\|_{L^2} \le \left\|\phi_1-\phi_2\right\|_{L^2}\exp\Big(c\int_0^T \left\|\nabla u_1\right\|_{L^{\infty}}+\left\|\nabla u_2\right\|_{L^{\infty}}dt\Big)
\end{equation*}
 from which uniqueness follows.

\subsection{Existence}

We mollify the initial datum as in the Bona--Smith argument~\cite{BS}.  Consider a radial and positive $\rho\in C_0^{\infty}(\mathbb{R}^d)$ such that $\rho(\xi)=1$ for $|\xi|\le 1/2$ and $\rho(\xi)=0$ for $|\xi|\ge1$, and define
$$u_{0,n}:=\big(\rho(n^{-1}\cdot)\widehat{u}_0\big)^{\vee}$$
for any integer $n\ge 1$.
First we state some properties of the regularized initial data.

\begin{lemma}\label{eqexist2}
Let $m\ge n \ge 1$ and $\alpha\le s$. Then
\begin{equation*}
    \left\|u_{0,n}-u_{0,m}\right\|_{H^{\alpha}}\lesssim n^{-(s-\alpha)}\|u_0\|_{H^{s}}.
\end{equation*}
Moreover, with $u_0\in H^s$,  
\begin{equation*}
    \left\|u_{0,n}-u_{0,m}\right\|_{H^{s}}\underset{n\to \infty}{\rightarrow} 0.
\end{equation*}
\end{lemma}
\begin{proof}
By support considerations, we observe
\begin{equation*}
    \begin{aligned}
    \left|\langle \xi\rangle^{\alpha}(\rho(\xi/m)-\rho(\xi/n))\widehat{u}_{0}(\xi)\right|^2 & \lesssim  n^{-2(s-\alpha)}\big|(\rho(\xi/m)-\rho(\xi/n))\big|^2   \left|\langle \xi\rangle^{s} \widehat{u}_0(\xi)\right|^2
        \end{aligned}
\end{equation*}
and so the first estimate follows by integrating and Plancherel's identity. With $\alpha=s$, the result follows by the Lebesgue dominated convergence theorem.
\end{proof}

Let $s\in (s_d,d+1]$ where $s_d := d/2+1/2$ for $d\ge 3$ and $s_2 := 5/3$ and take $\alpha\in (s_d,s]$. For integers $n\ge 1$, we consider solutions $u_n\in C([0,T];H^{d+1}(\mathbb{R}^{d}))$ of the higher dimensional Benjamin--Ono equation with mollified initial data $u_{0,n}$; 
\begin{equation}\label{RegulHBO}
 \left\{
\begin{aligned}
&\partial_t u_n-\mathcal{R}_1 \Delta u_n +u_n\partial_{x_1}u_{n}=0, \hskip 15pt x\in \R^d,\,  t\in \R, \\&u_n(x,0)=u_{0,n}(x).
\end{aligned}
\right.
\end{equation}
As  $\|u_{0,n}\|_{H^{\alpha}}\le \|u_{0,n}\|_{H^{s}}\le \|u_{0}\|_{H^{s}}$, from Lemma~\ref{apriESTlower}, we find that
\begin{equation}\label{eqexist3}
    \sup_{t\in[0,T]}\left\|u_n(t)\right\|_{H^\alpha} \le 2 \left\|u_0\right\|_{H^\alpha}, \quad \alpha\in(s_d,s]   \end{equation}
 whenever $T\le(1+A_s\left\|u_0\right\|_{H^s})^{-2}$,  and
\begin{equation}\label{eqexist4}
  K:=  \sup_{n\ge 1}
\int_0^T \left\|u_n(t)\right\|_{L^\infty}+\left\|\nabla u_n(t)\right\|_{L^{\infty}} dt <\infty.
\end{equation}

Now we set $v_{n,m}:=u_n-u_m$, so that $v_{n,m}$ satisfies 
\begin{equation}\label{eqexiscauchy}
    \partial_t v_{n,m}-\mathcal{R}_1 \Delta v_{n,m} +u_n\partial_{x_1}u_n-u_m\partial_{x_1}u_m=0,
\end{equation}
with initial datum $v_{n,m}(\cdot, 0)=u_{0,n}-u_{0,m}$.  Arguing as in the uniqueness Subsection~\ref{equniqu1} and using Lemma~\ref{eqexist2}, we deduce
\begin{equation*}
    \sup_{t\in[0,T]}\left\|v_{n,m}(t)\right\|_{L^2}  \le e^{cK} \left\|u_{0,n}-u_{0,m}\right\|_{L^2} \lesssim n^{-s}\|u_0\|_{H^{s}}.
\end{equation*}
Interpolating with~\eqref{eqexist3} we immediately get that
\begin{equation}\label{eqexist5}
    \sup_{t\in[0,T]}\left\|v_{n,m}(t)\right\|_{H^\alpha} \lesssim n^{-(s-\alpha)}\|u_0\|_{H^{s}}
\end{equation}
 so that $\left\{v_{n,m}\right\}$ is a Cauchy sequence in $C([0,T];H^\alpha(\mathbb{R}^d))$ for all $0\le\alpha< s$. 

Below we will find that $\left\{v_{n,m}\right\}$ is also Cauchy in $C([0,T];H^s(\mathbb{R}^d))$, but first we show that the sequence is Cauchy in $L^{1}([0,T];W^{1,\infty}(\mathbb{R}^d))$. In fact we prove something stronger that will help later.

\begin{lemma} \label{exislemma1} 
 Let $m\ge n \ge 1$. Then
 \begin{equation*}
 \int_0^T n\left\|v_{n,m}(t)\right\|_{L^\infty}+\left\|\nabla v_{n,m}(t)\right\|_{L^{\infty}} dt\underset{n\to \infty}{\rightarrow} 0.
\end{equation*}
\end{lemma}
\begin{proof}
 Let $\alpha\in(s_d,s)$ and rewrite the nonlinear term in \eqref{eqexiscauchy} as
 $$
 u_n\partial_{x_1}u_n-u_m\partial_{x_1}u_m=\tfrac{1}{2}\partial_{x_1}\big((u_n+u_m)v_{n,m}\big).
 $$ Taking  $w=\nabla v_{n,m}$ and $F=-\tfrac{1}{2}\nabla\partial_{x_1}\big((u_n+u_m)v_{n,m}\big)$, by Lemma~\ref{refinStri}, 
 \begin{align*}
 \int_0^T \|\nabla &v_{n,m}(t)\|_{L^{\infty}} dt\\
 &\lesssim T^{1/2} \left(\sup_{t\in[0,T]}\left\|v_{n,m}(\cdot,t)\right\|_{H^\alpha}+\int_0^T \left\|\partial_{x_1}\big((u_n+u_m)v_{n,m}\big)(\cdot,t) \right\|_{H^{\alpha-1}} dt\right)\\
&\lesssim T^{1/2} \left(\sup_{t\in[0,T]}\left\|v_{n,m}(\cdot,t)\right\|_{H^\alpha}+\int_0^T \|u_0\|_{H^\alpha}\|v_{n,m}(\cdot,t)\|_{H^\alpha}\, dt\right),
\end{align*}
where the second inequality follows from Lemma~\ref{conmKP} and Sobolev embedding. The desired convergence for this part then follows by applying \eqref{eqexist5}. 

On the other hand, taking $w=v_{n,m}$ and $F=-\tfrac{1}{2}\partial_{x_1}\big((u_n+u_m)v_{n,m}\big)$ in Lemma~\ref{refinStri}, we also have
 \begin{align*}
 \int_0^T \|&v_{n,m}(t)\|_{L^{\infty}} dt\\
&\lesssim T^{1/2} \left(\sup_{t\in[0,T]}\left\|v_{n,m}(\cdot,t)\right\|_{H^{\alpha-1}}+\int_0^T \|u_0\|_{H^{\alpha-1}}\|v_{n,m}(\cdot,t)\|_{H^{\alpha-1}}\, dt\right),
\end{align*}
and so we obtain an extra factor of $n^{-1}$  when applying \eqref{eqexist5}.
\end{proof}

\begin{prop}\label{exisprop1} 
Let $m\ge n \ge 1$. Then
\begin{equation*}
    \sup_{t\in[0,T]}\left\|v_{n,m}(t)\right\|_{H^s}\underset{n\to \infty}{\rightarrow} 0.
\end{equation*}
\end{prop}

\begin{proof} Applying the operator $J^s$ to~\eqref{eqexiscauchy}, rewriting the nonlinearity as$$u_n\partial_{x_1}u_n-u_m\partial_{x_1}u_m=v_{n,m}\partial_{x_1}u_n+u_m\partial_{x_1}v_{n,m},$$ and then multiplying the equation by $J^s v_{n,m}$ and integrating in space, we obtain
\begin{equation*}
    \begin{aligned}
    \frac{1}{2}\frac{d}{dt}\left\|J^s v_{n,m}\right\|_{L^2}^2    & =-\int_{\mathbb{R}^d} J^s\big(u_m\partial_{x_1}v_{n,m}\big)J^s v_{n,m} -\int_{\mathbb{R}^d} J^s\big(v_{n,m}\partial_{x_1}u_n\big)J^s v_{n,m} \\
    &=: -(A_1+A_2).
    \end{aligned}
\end{equation*}
Now by integrating by parts, we can write
\begin{equation*}
    \begin{aligned}
    A_1=\int_{\mathbb{R}^d}[J^s,u_m]\partial_{x_1}v_{n,m} J^sv_{n,m}-\frac{1}{2}\int_{\mathbb{R}^d} \partial_{x_1}u_m (J^sv_{n,m})^2,
    \end{aligned}
\end{equation*}
from which it follows from the Cauchy--Schwartz inequality and the commutator estimate of Lemma~\ref{conmKP} that
\begin{equation*}
    \begin{aligned}
    |A_1| \lesssim \left\|\nabla u_m\right\|_{L^{\infty}}\left\|J^ sv_{n,m}\right\|_{L^2}^2.
        \end{aligned}
\end{equation*}
On the other hand, we can write
\begin{equation*}
     \begin{aligned}
    A_2=\int_{\mathbb{R}^d}[J^s,v_{n,m}]\partial_{x_1}u_{n} J^sv_{n,m}+\int_{\mathbb{R}^d} v_{n,m}(J^s\partial_{x_1}u_n) J^sv_{n,m}.
    \end{aligned}
\end{equation*}
Again by the Cauchy--Schwarz inequality and the commutator estimate of Lemma~\ref{conmKP},   
\begin{equation*}
    \begin{aligned}
    |A_2|& \lesssim  \left\|\nabla u_n\right\|_{L^\infty} \left\|J^s v_{n,m} \right\|_{L^2}^2+\left\|\nabla v_{n,m}\right\|_{L^\infty} \left\|J^s u_{n} \right\|_{L^2}\left\|J^s v_{n,m}\right\|_{L^2} \\
    &\hspace{0.5cm}+\left\|v_{n,m}\right\|_{L^{\infty}}\left\|J^{s+1} u_{n}\right\|_{L^2}\left\|J^s v_{n,m}\right\|_{L^2}.
    \end{aligned}
\end{equation*}
Now arguing as in the proof of Lemma~\ref{apriEST}, we have
\begin{equation*}
    \frac{d}{dt}\left\|J^{s+1}u_n\right\|_{L^2}^2 \lesssim \left\|\nabla u_n\right\|_{L^{\infty}}\left\|J^{s+1}u_n\right\|_{L^2}^2,
\end{equation*}
so, by Gronwall's inequality, 
\begin{equation*}
\begin{aligned}
\left\|J^{s+1}u_n\right\|_{L^2} &\le e^{cK}\left\|J^{s+1}u_{0,n}\right\|_{L^2} \lesssim n \left\|u_{0}\right\|_{H^s}.
\end{aligned}
\end{equation*}
where $K$ is defined as in~\eqref{eqexist4}. In view of ~\eqref{eqexist3}, we can bound both $\left\|J^s u_{n}\right\|_{L^2}$ and $\left\|J^s v_{n,m}\right\|_{L^2}$ by a constant multiple of $\left\|u_0\right\|_{H^s}$, so that
\begin{equation*}
\begin{aligned}
|A_2| \lesssim \left\|\nabla u_n\right\|_{L^{\infty}}\left\|J^s v_{n,m}\right\|_{L^2}^2+\big( n\left\|v_{n,m}\right\|_{L^{\infty}}+\left\|\nabla v_{n,m}\right\|_{L^{\infty}}\big)\left\|u_{0}\right\|_{H^s}^2.
\end{aligned}
\end{equation*}

Summing up our estimates for $A_1$ and $A_2$, we find that
\begin{equation}
\begin{aligned}
\frac{d}{dt}\left\|J^s v_{n,m}(t)\right\|_{L^2}^2 & \le  
a(t)\left\|J^s v_{n,m}(t)\right\|_{L^2}^2 +b(t)\end{aligned}
\end{equation}
where 
\begin{align*}
   a(t)&:=C_0 \Big(\left\|\nabla u_n(t)\right\|_{L^{\infty}}+\left\|\nabla u_m(t)\right\|_{L^{\infty}}\Big),\\
    b(t)&:=C_1\Big(n \left\|v_{m,m}(t)\right\|_{L^{\infty}}+ \left\|\nabla v_{n,m}(t)\right\|_{L^{\infty}}\Big) \left\|u_{0}\right\|_{H^s}^2.
\end{align*}
Now if $g(t)$ solves
\begin{equation*}
    \left\{\begin{aligned}
    &\frac{d}{dt} g(t) = a(t)g(t)+b(t), \\
    &g(0)=\left\|u_{0,n}-u_{0,m}\right\|_{H^s}^2,
    \end{aligned}\right.
\end{equation*}
then
\begin{equation*}
    \begin{aligned}
    \frac{d}{dt}\left( \left\|J^s v_{n,m}(t)\right\|_{L^2}^2 - g(t)\right)\le a(t)\left(\left\|J^s v_{n,m}(t)\right\|_{L^2}^2 - g(t)\right)
    \end{aligned}
\end{equation*}
with initial condition, $\left\|J^s v_{n,m}(0)\right\|_{L^2}^2 - g(0)=0$. Then by an application of Gronwall's inequality, we find that  $\left\|J^s v_{n,m}(t)\right\|_{L^2}^2 \le g(t)$ for all $t \ge 0$. 
Now as $g(t)$ has the explicit form
\begin{equation*}\label{expli}
    g(t)=g(0)e^{\int_0^t a(t')\, dt'}+\int_0^t b(\tau)e^{\int_{\tau}^t a(t') \, dt'} \, d\tau,
\end{equation*}
we find that $ \sup_{t\in[0,T]} \left\|J^s v_{n,m}(t)\right\|^2_{L^2}$ is bounded by 
\begin{equation*}
  e^{cK}\Big(\left\|u_{0,n}-u_{0,m}\right\|_{H^s}^2+\left\|u_{0}\right\|_{H^s}^2\int_0^T n\left\|v_{n,m}(t)\right\|_{L^\infty}+\left\|\nabla v_{n,m}(t)\right\|_{L^{\infty}} dt\Big) \underset{n\to\infty}{=}0,
\end{equation*}
where the convergence follows from Lemmas~\ref{eqexist2} and \ref{exislemma1}.
\end{proof}

We deduce from Proposition \ref{exisprop1} and Lemma \ref{exislemma1} that $u_n$ has a limit  $u$ in 
$$C\big([0,T];H^s(\mathbb{R}^d))\cap L^1([0,T];W^{1,\infty}(\mathbb{R}^d)\big).$$ 
Now recalling that \begin{equation}\label{Intequ}
    u_n(t)=e^{t \mathcal{R}_1 \Delta }u_{0,n}-\frac{1}{2}\int_{0}^t e^{(t-t') \mathcal{R}_1 \Delta }\partial_{x_1}[u_n(t')]^2\, dt',
\end{equation}
and noting that
\begin{equation*}
    \begin{aligned}
    &\left\|\int_0^t e^{(t-t')\mathcal{R}_1 \Delta }\partial_{x_1}\left[u_n(t')^2-u(t')^2\right]\, dt'\right\|_{H^{s-1}} \\
    &\hspace{3cm}\le \int_0^t \left\|u_n(t')^2-u(t')^2\right\|_{H^s} \, dt' \\
    & \hspace{3cm} \lesssim \int_0^t \left\|u_n(t')+u(t')\right\|_{H^s}\left\|u_n(t')-u(t')\right\|_{H^s} \, dt',
    \end{aligned}
\end{equation*}
we see that $u$ also solves the integral formulation of \eqref{HBO-IVP} in the  $C([0,T];H^{s-1}(\mathbb{R}^d))$ sense.


\subsection{Continuity of the flow map.} Let $s>s_d$ where $s_d := d/2+1/2$ for $d\ge3$ and $s_2 := 5/3$.
Fix  $u_0\in H^s$ and $t<T=T(\left\|u_0\right\|_{H^s})$. We are required to prove that for all $\epsilon>0$,  there exists $\delta>0$ such that for all data $v_0$ such that $\left\|u_0-v_0\right\|_{H^s}<\delta$, we have
\begin{equation}\label{eqcontdp1}
    \left\|u(t)-v(t)\right\|_{H^s}<\epsilon.
\end{equation}
For $n\ge 1$, we regularize the initial data as in the previous section. Then from the triangle inequality we obtain 
\begin{equation*}
    \left\|u(t)-v(t)\right\|_{H^s}\le \left\|u(t)-u_n(t)\right\|_{H^s}+\left\|u_n(t)-v_n(t)\right\|_{H^s}+\left\|v_n(t)-v(t)\right\|_{H^s}.
\end{equation*}
By the definitions of $u$ and $v$,  we can take $n$ sufficiently large so that
\begin{equation}\label{eqcontdp2}
 \left\|u(t)-u_n(t)\right\|_{H^s}+\left\|v_n(t)-v(t)\right\|_{H^s} < \epsilon/2.
\end{equation}
On the other hand, \begin{equation*}
    \left\|u_{0,n}-v_{0,n}\right\|_{H^{d+1}}\lesssim n^{d+1-s} \left\|u_0-v_0\right\|_{H^s} \le n^{d+1-s} \delta. 
\end{equation*}
Then using the continuity of the flow map for smooth solutions, we can choose $\delta>0$ small enough to ensure
\begin{equation}\label{eqcontdp3}
   \left\|u_{n}(t)-v_{n}(t)\right\|_{H^{s}} \le  \left\|u_{n}(t)-v_{n}(t)\right\|_{H^{d+1}} \le \epsilon/2.
\end{equation}
Therefore estimate~\eqref{eqcontdp1} follows by combining~\eqref{eqcontdp2} and~\eqref{eqcontdp3}.


\section{Appendix: Ill-posedness results}

Here we prove that~\eqref{HBO-IVP}  cannot be solved in $H^s(\mathbb{R}^d)$ by a Picard iterative scheme based on the Duhamel formula. This result can be viewed as an extension of~\cite{molin}, where the $C^2$ ill-posedness in $H^s(\mathbb{R})$ is established for the Benjamin--Ono equation.

\begin{proof}[Proof of Theorem~\ref{illpossed}]
Suppose that there exists $T>0$ such that ~\eqref{HBO-IVP} is locally well-posed in $H^s(\mathbb{R}^d)$ on the time interval $[0,T)$ and such that the flow map 
$$\Phi(t): H^s(\mathbb{R}^d)\to H^s(\mathbb{R}^d), \hspace{0.2cm} u_0 \mapsto u(t)$$ 
is $C^2$ differentiable at the origin. When $\phi \in H^s(\mathbb{R}^d)$, we have that $\Phi(\cdot)\phi$ is a solution of~\eqref{HBO-IVP} with initial data $\phi$ so by Duhamel's principle $\Phi(t)\phi$ must satisfy the integral equation
$$\Phi(t)\phi=e^{t \mathcal{R}_1 \Delta }\phi-\frac{1}{2}\int_{0}^t e^{(t-t') \mathcal{R}_1 \Delta }\partial_{x_1}\big[\Phi(t')\phi\big]^2\, dt'.$$
We compute the Fr\' echet derivative of $\Phi(t)$ at $\psi$ with direction $\phi_1$,
\begin{equation}\label{Fder3}
    d_{\psi}\Phi(t)(\phi_1)=e^{t \mathcal{R}_1 \Delta}\phi_1- \int_{0}^t e^{(t-t') \mathcal{R}_1 \Delta }\partial_{x_1}\big[\Phi(t')\psi \, d_{\psi}\Phi(t')(\phi_1)\big]\, dt'.
\end{equation}
Supposing that~\eqref{HBO-IVP} is well-posed, uniqueness implies that $\Phi(t)(0)=0$, so that $d_{0}\Phi(t)(\phi_1)=e^{t \mathcal{R}_1 \Delta}\phi_1$.
Differentiating again we find that
\begin{equation}
    \begin{aligned}\nonumber
    d_0^2\Phi(t)(\phi_1,\phi_2)&=\left. \frac{\partial}{\partial \gamma}\Big(\gamma \mapsto d_{\gamma \phi_2}\Phi(t)(\phi_1)\Big) \right|_{\gamma=0}\\
    &=-\left.\int_0^t e^{(t-t')\mathcal{R}_1\Delta} \partial_{x_1}\big[ d_{\gamma \phi_2}\Phi(t)(\phi_2)d_{\gamma \phi_2}\Phi(t)(\phi_1)\big] \, dt'\right|_{\gamma=0} \\
    &\quad-\left.\int_0^t e^{(t-t')\mathcal{R}_1\Delta} \partial_{x_1}\big[ \Phi(t)(\gamma \phi_2)d^2_{\gamma \phi_1}\Phi(t)(\phi_1,\phi_2)\big] \, dt' \right|_{\gamma=0},
    \end{aligned}
\end{equation}
so that
\begin{align*}
d^2_0\Phi(t)(\phi_1,\phi_2) =- \int_0^t e^{(t-t')\mathcal{R}_1 \Delta}\partial_{x_1}\big[(e^{t'\mathcal{R}_1 \Delta} \phi_1)(e^{t'\mathcal{R}_1 \Delta} \phi_2)\big]\, dt'. \label{Fder2}
\end{align*}
Now, if the flow map were $C^2$ then $d^2_0\Phi(t)$ would be bounded from $H^s\times H^s$ to $H^s$;
\begin{equation*}
\left\|\int_0^t e^{(t-t')\mathcal{R}_1 \Delta}\partial_{x_1}\big[(e^{t'\mathcal{R}_1 \Delta} \phi_1)(e^{t'\mathcal{R}_1 \Delta}  \phi_2)\big]\ dt'\right\|_{H^s}\lesssim \left\|\phi_1 \right\|_{H^s}\left\|\phi _2\right\|_{H^s}.
\end{equation*}
We will prove that this does not hold in general, following the arguments in~\cite{molin}.

Indeed, we will construct two sequences of functions, $\phi_{1,N}$ and $\phi_{2,N}$, such that
\begin{equation}\label{ILP2}
    \left\|\phi_{1,N}\right\|_{H^s}, \left\|\phi_{2,N}\right\|_{H^s} \le C
\end{equation}
and 
\begin{equation}\label{ILPP2}
    \lim_{N\to \infty} \left\|\int_0^t e^{(t-t')\mathcal{R}_1 \Delta} \partial_{x_1}\big[(e^{t'\mathcal{R}_1 \Delta} \phi_{1,N})(e^{t'\mathcal{R}_1 \Delta}  \phi_{2,N})\big]\ dt'\right\|_{H^s}=\infty.
\end{equation}
We define $\phi_{1,N}$ and $\phi_{2,N}$ via their Fourier transforms as
\begin{equation*}
    \left\{\begin{aligned}
    &\widehat{\phi_{1,N}}(\xi)=\lambda^{\frac{1-2d}{2d}}N^{-s}\chi_{D_1}(\xi), &&\text{ with } D_1=[N,N+\lambda]\times [\lambda^{1/d}/2,\lambda^{1/d}]^{d-1}, \\
    &\widehat{\phi_{2,N}}(\xi)=\lambda^{\frac{1-2d}{2d}}\chi_{D_2}(\xi), &&\text{ with } D_2=[3\lambda,4\lambda]\times [\lambda^{1/d}/2,\lambda^{1/d}]^{d-1} 
    \end{aligned}\right.
\end{equation*}
where $N\gg 1$, $\lambda=N^{-(1+\epsilon)}$ and  $0<\epsilon<1/(2d-1)$. First, we observe that $\phi_{1,N}$ and $\phi_{2,N}$ satisfy~\eqref{ILP2}.

On the other hand, taking the Fourier transform with respect to the space variable,
\begin{equation}\label{ILP3}
    \begin{aligned}
     \widehat{I_N}(\xi,t)&:=\left\{\int_0^t e^{(t-t')\mathcal{R}_1 \Delta} \partial_{x_1} \big[(e^{t'\mathcal{R}_1 \Delta} \phi_{1,N})(e^{t'\mathcal{R}_1 \Delta}  \phi_{2,N})\big]\ dt'\right\}^{\wedge}(\xi)\\
    &=\int_{K_{\xi}} \xi_1 e^{it\xi_1|\xi|}\frac{e^{i\sigma(\xi,\eta)t}-1}{\sigma(\xi,\eta)} \widehat{\phi_{1,N}}(\eta)\widehat{\phi_{2,N}}(\xi-\eta)\, d\eta 
    \end{aligned}
\end{equation}
where the resonant function is given by
\begin{equation*}
\begin{aligned}
\sigma(\xi,\eta):=-\xi_1|\xi|+(\xi_1-\eta_1)|\xi-\eta|+\eta_1|\eta|
\end{aligned}
\end{equation*}
and
\begin{equation*}
    K_{\xi}:=\left\{\eta\in \mathbb{R}^d\, : \, \eta\in D_1, \, \xi-\eta\in D_2 \right\}.
\end{equation*}
When $\eta \in D_1$ and $\xi-\eta \in D_2$, we claim that 
\begin{equation}\label{ILP4}
    |\sigma(\xi,\eta)| \sim \lambda N.
\end{equation}
Indeed, using that $\widehat{I}_N(\xi)$ is supported on 
$$D_3= [N+3\lambda,N+5\lambda]\times[\lambda^{1/d},2\lambda^{1/d}]^{d-1}$$
we easily obtain
\begin{equation}\label{ILP5}
    (\xi_1-\eta_1)|\xi-\eta|\sim \lambda^{(d+1)/d}.
\end{equation}
Moreover, from the inequality
$$|\xi|\le \Big((N+5\lambda)^2+4(d-1)\lambda^{2/d}\Big)^{1/2}  \le N+6\lambda$$
which holds for $N$ large, $\lambda=N^{-(1+\epsilon)}$ with $0<\epsilon<1/(2d-1)$, we have
\begin{equation}\label{ILP6}
\begin{aligned}
    (N+3\lambda)^2\le \xi_1|\xi|\le (N+6\lambda)^2.
\end{aligned}    
\end{equation}
Analogously, we get
\begin{equation}\label{ILP7}
\begin{aligned}
    N^2\le \eta_1|\eta|  \le (N+2\lambda)^2.
\end{aligned}    
\end{equation}
Then,~\eqref{ILP4} follows from~\eqref{ILP5},~\eqref{ILP6} and~\eqref{ILP7}. 

Now, since $\lambda N=N^{-\epsilon}$ and $|\sigma(\xi,\eta)|\sim \lambda N$ it follows 
\begin{equation}\label{ILP8}
\left|\frac{e^{i\sigma(\xi,\eta)t}-1}{\sigma(\xi,\eta)}\right|= |t|+O\left(\frac{1}{N^{\epsilon}}\right).
\end{equation}
From~\eqref{ILP8} and $|K_\xi|\sim \lambda^{(2d-1)/d}$, we infer that
\begin{align*}
|\widehat{I_N}(\xi,t)|\chi_{D_3}(\xi) & \gtrsim \frac{N \lambda^{(2d-1)/d}}{N^{s}\lambda^{(2d-1)/d}}|t| \, \chi_{D_3}(\xi).
\end{align*}
Therefore we arrive at
\begin{equation*}
\left\|I_N(t)\right\|_{H^s} \gtrsim N \lambda^{(2d-1)/2d}|t|=N^{1/2d-\epsilon((2d-1)/2d)}|t|.
\end{equation*}
Now as $0<\epsilon<1/(2d-1)$, from this we deduce~\eqref{ILPP2}, which completes the proof.
\end{proof}

The following corollary (of the proof) shows that it is not possible to solve \eqref{HBO-IVP} in $H^s(\mathbb{R}^d)$ via the usual contraction argument.

\begin{cor}\label{illrem1} Let $s\in \mathbb{R}$ and $T>0$. Then there does not exist a space $X_T$ continuously embedded in $C([0,T];H^s(\mathbb{R}^d))$ such that
\begin{equation}\label{ILP9}
    \left\|e^{t \mathcal{R}_1 \Delta }\phi\right\|_{X_T}\le C \left\|\phi \right\|_{H^s}
\end{equation}
and 
\begin{equation}\label{desire}
\left\|\int_0^t e^{(t-t')\mathcal{R}_1 \Delta}\big[F(\cdot,t')\partial_{x_1}F(\cdot, t')\big]\ dt'\right\|_{X_T}\le C\left\|F(\cdot, t)\right\|_{X_T}^2.
\end{equation}
\end{cor}

\begin{proof} We write $F=F_1+F_2$ and note that
\begin{equation*}
\begin{aligned}
\Big\|&\int_0^t e^{(t-t')\mathcal{R}_1 \Delta}\big[F\partial_{x_1}F(\cdot,t')\big]\ dt'\Big\|_{X_T} 
\ge\ \left\|\int_0^t e^{(t-t')\mathcal{R}_1 \Delta}\partial_{x_1}\big[F_1F_2(\cdot,t')\big]\ dt'\right\|_{X_T} \\
&-\left\|\int_0^t e^{(t-t')\mathcal{R}_1 \Delta}\big[F_1\partial_{x_1}F_1(\cdot,t')\big]\, dt'\right\|_{X_T} -\left\|\int_0^t e^{(t-t')\mathcal{R}_1 \Delta}\big[F_2\partial_{x_1}F_2(\cdot,t')\big]\, dt'\right\|_{X_T}.
\end{aligned}
\end{equation*}
Now taking $F_1(\cdot,t'):=e^{t'\mathcal{R}_1 \Delta} \phi_{1,N}$ and $F_2(\cdot,t'):=e^{t'\mathcal{R}_1 \Delta} \phi_{1,N}$, by \eqref{ILP9} and \eqref{ILP2}, we have
$$
\|F\|_{X_T},\quad \|F_1\|_{X_T},\quad \|F_2\|_{X_T} \le C.
$$
Thus, if \eqref{desire} held, we would find that
\begin{equation*}
   \left\|\int_0^t e^{(t-t')\mathcal{R}_1 \Delta} \partial_{x_1}\big[(e^{t'\mathcal{R}_1 \Delta} \phi_{1,N})(e^{t'\mathcal{R}_1 \Delta}  \phi_{2,N})\big]\ dt'\right\|_{X_T}
\end{equation*}
is uniformly bounded in $N$, contradicting \eqref{ILPP2}.
\end{proof}


Next we prove that the flow map could not be uniformly continuous in $L^2(\mathbb{R}^2)$. We recall that Mari\c s~\cite{M} proved that there exists solitary wave solutions of the form $u_c(x_1,x_2,t)=\varphi(x_1-ct,x_2)$ with $c>0$.
That is to say, $\varphi_c$ is a solution of the time independent equation
\begin{equation}\label{statioWS}
    -c\partial_{x_1}\varphi-\mathcal{R}_1\Delta \varphi+\varphi\partial_{x_1}\varphi=0
\end{equation}
where $\varphi_c \in H^s(\mathbb{R}^2)$ for all $s\ge 0$.

\begin{proof}[Proof of Proposition~\ref{illpossedl2}]
Let $\varphi_c(x_1,x_2):=c\varphi_1(cx_1,cx_2)$ where $\varphi_1$ solves~\eqref{statioWS} with $c=1$. Then $\varphi_c$ solves~\eqref{statioWS} with $c>0$ and we consider solutions
$$u_c(x_1,x_2,t):=c\varphi_1(cx_1-c^2t,c x_2)$$
to~\eqref{HBO-IVP}. In particular we will consider solutions $u_{c_1}$ and $u_{c_2}$ with $c_1 \neq c_2$. 

By a change of variables it is easy to see that, for all $t>0$,
$$\left\|u_{c_1}(\cdot,t)\right\|_{L^2}=\left\|\varphi_1\right\|_{L^2}=\left\|u_{c_2}(\cdot,t)\right\|_{L^2},$$
so that
\begin{equation}\label{apeq1}
    \begin{aligned}
    \left\|u_{c_1}(\cdot,t)-u_{c_2}(\cdot,t)\right\|_{L^2}^2=2\left\|\varphi_{1}\right\|_{L^2}^2-2\langle u_{c_1}(\cdot,t), u_{c_2}(\cdot,t) \rangle_{L^2}.
    \end{aligned}
\end{equation}
Changing variables by $c_2 x_1 -c_2^2 t\rightarrow x_1$ and $ c_2 x_2\rightarrow x_2$, we see that
\begin{equation*}
    \begin{aligned}
    \big\langle u_{c_1}(\cdot,t), u_{c_2}(\cdot,t) \big\rangle_{L^2}=\frac{c_1}{c_2} \int \varphi_1\big(\tfrac{c_1}{c_2}  (x_1-c_2(c_1-c_2)t),\tfrac{c_1}{c_2} x_2\big)\overline{\varphi_1(x)} \, dx.
    \end{aligned}
\end{equation*}
Therefore, taking $c_1=n+1$, $c_2=n$, from the Lebesgue dominated convergence theorem, it follows  that, for all  $t>0$, 
$$\big\langle u_{c_1}(\cdot,t), u_{c_2}(\cdot,t) \big\rangle_{L^2}= \frac{c_1}{c_2} \int \varphi_1\big(\tfrac{c_1}{c_2}(x_1-nt,x_2)\big) \overline{\varphi_1(x)} \, dx \to 0 \, \text{ as } \, n \to \infty,$$
while 
$$\big\langle u_{c_1}(\cdot,0), u_{c_2}(\cdot,0) \big\rangle_{L^2} \to \left\|\varphi_1\right\|_{L^2}^2 \, \text{ as } \, n \to \infty.$$
Thus, in view of~\eqref{apeq1}, we deduce
\begin{equation*}
    \left\|u_{c_1}(\cdot,0)-u_{c_2}(\cdot,0)\right\|_{L^2} \to 0 \, \text{ as } \, n \to \infty,
\end{equation*}
while on the other hand, for all $t>0$,
\begin{equation*}
    \left\|u_{c_1}(\cdot,t)-u_{c_2}(\cdot,t)\right\|_{L^2} \to 2^{1/2}\left\|\varphi_1\right\|_{L^2} \, \text{ as } \, n \to \infty,
\end{equation*}
completing the proof.
\end{proof}


\end{document}